\documentclass[12pt]{amsart}

\usepackage[margin = 1in]{geometry}
\usepackage{amsfonts, amsmath,amscd, amssymb, cite, latexsym, mathrsfs, mathtools,
slashed, stmaryrd, verbatim, wasysym }
\usepackage[all]{xy}
\usepackage[pdftex]{graphicx}
\usepackage{hyperref}
\usepackage[applemac]{inputenc}

\newtheorem*{acknowledgements}{Acknowledgements}
\newtheorem{theorem}{Theorem}
\newtheorem{corollary}[theorem]{Corollary}
\newtheorem{definition}{Definition}

\newtheorem{lemma}[theorem]{Lemma}
\newtheorem{proposition}[theorem]{Proposition}
\newtheorem{remark}{Remark}

\numberwithin{equation}{section}
\numberwithin{theorem}{section}

\let\oldsqrt\sqrt
\def\sqrt{\mathpalette\DHLhksqrt}
\def\DHLhksqrt#1#2{%
\setbox0=\hbox{$#1\oldsqrt{#2\,}$}\dimen0=\ht0
\advance\dimen0-0.2\ht0
\setbox2=\hbox{\vrule height\ht0 depth -\dimen0}%
{\box0\lower0.4pt\box2}}

\renewcommand{\Re}{\operatorname{Re}}
\renewcommand{\hat}[1]{\widehat{#1}}

\newcommand{\wt}[1]{\widetilde{#1}}
\newcommand{\rest}[1]{\big\rvert_{#1}} 

\newcommand\lra{\longrightarrow}

\newcommand\pa{\partial}

\newcommand\cf{cf\@. }

\newcommand\eps\varepsilon
\renewcommand\epsilon\varepsilon

\newcommand\hc{\operatorname{hc}}

\newcommand\CI{{\mathcal{C}}^{\infty}}

\renewcommand\det{\operatorname{det}}

\newcommand\dvol{\operatorname{dvol}}

\newcommand\Id{\operatorname{Id}}

\renewcommand\Re{\operatorname{Re}}

\newcommand\Spec{\operatorname{Spec}}

\newcommand\supp{\operatorname{supp}}

\newcommand\Tr{\operatorname{Tr}}

\newcommand\Vol{\operatorname{Vol}}
\newcommand\dist{\operatorname{dist}}
\newcommand\Fcb{\operatorname{Fcb}}
\newcommand\pr{\operatorname{pr}}
\newcommand\horn{\operatorname{hc}}
\newcommand\HS{\operatorname{HS}}

\newcommand\Mforall{\text{ for all }}

\newcommand\Mif{\text{ if }}

\newcommand\Mor{\text{ or }}

\newcommand\Mwith{\text{ with }}

\newcommand\paperintro%
        {%
         }
\newcommand\paperbody%
        {%
         }

\newcommand\bbC{\mathbb{C}}

\newcommand\bbH{\mathbb{H}}

\newcommand\bbN{\mathbb{N}}

\newcommand\bbR{\mathbb{R}}
\newcommand\bbS{\mathbb{S}}

\newcommand\bbZ{\mathbb{Z}}

\newcommand\cS{\mathcal{S}}

\newcommand\cU{\mathcal{U}}
\newcommand\cV{\mathcal{V}}
\newcommand\cW{\mathcal{W}}

\newcommand\hW{\widehat{\cW}}
\newcommand\ty{\widetilde{y}}
\newcommand\td{\widetilde{d}}
\newcommand\tG{\widetilde{G}}

\DeclareMathAlphabet{\mathpzc}{OT1}{pzc}{m}{it}

\begin{document}

\title[Compactness via conformal surgeries]{Compactness of  relatively isospectral sets of surfaces via conformal surgeries}
\author{Pierre Albin}
\author{Clara L. Aldana}
\author{Fr\'ed\'eric Rochon}
\address{Department of Mathematics, University of Illinois at Urbana-Champaign}
\email{palbin@illinois.edu}
\address{Max Planck Institut f\"ur Gravitationsphysik (AEI)}
\email{clara.aldana@aei.mpg.de}
\address{Département de Mathématiques, UQÀM}
\email{rochon.frederic@uqam.ca }
\thanks{PA was partially funded by NSF grant DMS-1104533.\\ \indent
CA was partially funded by FCT project ptdc/mat/101007/2008. Registered at MPG, AEI-2012-059.\\ \indent
FR was supported by grant DP120102019}

\begin{abstract}
We introduce a notion of relative isospectrality for surfaces with boundary having possibly non-compact ends either conformally compact or asymptotic to cusps.  We obtain a compactness result for such families via a conformal surgery that allows us to reduce to the case of surfaces hyperbolic near infinity recently studied by Borthwick and Perry, or to the closed case by Osgood, Phillips and Sarnak if there are only cusps.
\end{abstract}

\maketitle

\tableofcontents

\paperintro
\section*{Introduction}

Despite the negative answer to Mark Kac's famous question
\begin{quote} Can one hear the shape of a drum? \end{quote}
there have been many interesting positive results.
The line most relevant to our current investigation starts with a seminal
paper by Melrose \cite{Melrose:Drumheads}.
For bounded planar domains, or `drumheads',  the geometry is encoded in
the geodesic curvature of the boundary and Melrose showed that
the short-time asymptotics of the trace of the heat kernel
determines this curvature to within a compact subset of the space
of smooth functions.

A stronger compactness result was obtained by Osgood, Phillips, and Sarnak
\cite{OPS1, OPS2, OPS3, OPS4} by making use of the determinant of the Laplacian,
in addition to the short-time asymptotics of the heat trace.
Indeed, they were able to show that on a given closed surface a family of
isospectral metrics is compact in the space of smooth metrics (and similarly
for planar domains). Their proof can be conveniently understood in terms of
the Cheeger compactness theorem:
a set of metrics with lower bounds on the injectivity radius and volume, and
uniform pointwise upper bounds on the curvature and its covariant derivatives is compact.
The short-time asymptotics of the heat trace determine the volume and upper bounds
on the Sobolev norms of the curvature \cite{Melrose:Drumheads, OPS2, Gilkey:Leading, Brooks-Glezen}
while for surfaces the determinant of the Laplacian gives a lower bound on the injectivity radius
\cite{OPS2}. The control on the injectivity radius and the Sobolev norms of the curvature can then be parlayed into uniform pointwise bounds on the curvature and its covariant derivatives.

There have been many extensions of the Osgood-Phillips-Sarnak results.  For instance, it is only recently that their result for planar domains has been extended to flat surfaces with boundary by Kim \cite{Kim2008}.
In higher dimensions there are results of Brooks-Perry-Yang \cite{Brooks-Perry-Yang}, Chang-Yang \cite{Chang-Yang}, and Chang-Qing \cite{Chang-Qing} studying compactness of isospectral metrics in a given conformal class, as well as results establishing compactness of isospectral metrics with an additional assumption on the curvature or the injectivity radius, e.g., by Anderson \cite{Anderson}, Brooks-Glezen \cite{Brooks-Glezen}, Brooks-Perry-Petersen \cite{Brooks-Perry-Petersen}, Chen-Xu \cite{Chen-Xu}, and Zhou \cite{Zhou}. For a well-written survey of positive and negative isospectral results we refer the reader to \cite{Gordon-Perry-Schueth}.

The first extension of these compactness results to non-compact surfaces
is the paper by Hassell and Zelditch \cite{Hassell-Zelditch}.
These authors consider exterior domains on $\bbR^2,$ i.e., each of their
surfaces is the exterior of a compact subset of $\bbR^2$ and they
consider the Laplacian with Dirichlet boundary conditions.
The spectrum of such a Laplacian is always equal to $[0,\infty),$
so they propose a replacement of the isospectrality condition,
isophasality. For manifolds that coincide with $\bbR^n$ outside a compact
set, or more generally for non-compact manifolds with an asymptotically
regular structure at infinity, one can define a scattering operator
$\lambda \mapsto \cS(\lambda)$
(see e.g., \cite{Melrose:Scattering})
and the scattering phase is defined to be $s(\lambda) = -i\log\det \cS(\lambda).$
In the context of exterior domains, the scattering phase is a natural replacement
for the counting function of the spectrum, and requiring that two manifolds have
the same scattering phase is a natural replacement for isospectrality.
Hassell and Zelditch show that a family of isophasal exterior domains is compact
in the space of smooth domains.

More recently, Borthwick and Perry \cite{Borthwick-Perry} consider non-compact
surfaces whose ends are hyperbolic funnels.
For these metrics the resolvent $R(s) = (\Delta - s(1-s))^{-1}$ admits a meromorphic
continuation to the whole complex plane \cite{Mazzeo-Melrose:Zero, Guillarmou:Mero, Guillope-Zworski}.
Two metrics whose resolvents have the same poles (with multiplicity) are called
isoresonant.
Borthwick and Perry prove that any set of isoresonant metrics that coincide,
and are hyperbolic funnel metrics, in a fixed
`neighborhood of infinity' form a $\CI$-compact subset in the space of metrics.
This generalizes earlier work of Borthwick, Judge, and Perry \cite{Borthwick-Judge-Perry:DetLap}.

One of the authors \cite{Aldana:Isoresonant} has shown that metrics that are conformally equivalent to a hyperbolic surface with cusps, with conformal factors supported in a fixed compact set, and mutually isoresonant form a $\CI$-compact set. Note that due to the vanishing of the injectivity radius, one cannot use Sobolev inequalities to transform Sobolev bounds on the curvature to uniform bounds. We shall face the same problem below (see, e.g., Lemma~\ref{del.1}).

The results in the noncompact setting are similar in the sense that the metrics are assumed to coincide outside of a compact set, but they differ in the definition of isospectrality.
We propose a notion of isospectrality for noncompact manifolds that unifies these approaches.
Let us say that two Riemannian manifolds $(M_1,g_1)$ and $(M_2,g_2)$
{\em coincide cocompactly} if there exist compact subsets $K_i \subseteq M_i\setminus \pa M_i$ and an isometry $M_1 \setminus K_1 \lra M_2 \setminus K_2.$
In this case, we say that the manifolds coincide on $\cU_\infty= M_i \setminus K_i.$
By embedding each $L^2(M_i,g_i)$ into
\begin{equation*}
    L^2(K_1, g_1) \oplus L^2(K_2, g_2) \oplus L^2(\cU_\infty, g_i),
\end{equation*}
we can consider the difference of the heat kernels
\begin{equation*}
    e^{-t\Delta_{g_1}} - e^{-t\Delta_{g_2}}.
\end{equation*}
Often this difference is trace-class even if the individual heat kernels are not.
Indeed, results of Bunke \cite{Bunke1992} and Carron \cite{Carron:DetRel} guarantee that for complete metrics the difference of heat kernels is trace-class.

\begin{definition}
Two Riemannian manifolds $(M_1,g_1)$ and $(M_2,g_2)$ are {\bf relatively isospectral} if they coincide on $M_i\setminus K_i$ with $K_i\subset M_i\setminus \pa M_i$ compact sets and if their relative heat trace is (defined and) identically zero,
\begin{equation}
    \Tr( e^{-t\Delta_1}- e^{-t\Delta_2} ) =0 \quad \Mforall t >0.
\label{RelIsoCond}\end{equation}
We also say that they are isospectral relative to $\cU_\infty$ to emphasize the neighborhood where they coincide.
\end{definition}

Notice that, on a closed manifold, two metrics are isospectral if and only if the traces of their heat kernels coincide for all positive time, motivating this definition.
Indeed asking that the spectrum of two Laplacians coincide with multiplicity
is the same as asking that the trace of their spectral measures coincide as
measures on the positive real line. Taking Laplace transforms this is equivalent
to asking that the trace of the heat kernels coincide.


If the manifolds are complete, then \eqref{RelIsoCond} can be interpreted as a condition on the
`Krein spectral shift function' $\xi$ of the pair $\Delta_1,$ $\Delta_2,$ as one has, e.g., \cite[Theorem 3.3]{Carron:DetRel}\footnote{The function $\xi$ is closely related to the `scattering phase', see \cite{Carron:DetRel}.}
\begin{equation}\label{RelKrein}
    \Tr( e^{-t\Delta_1}- e^{-t\Delta_2} ) =  - \int_{\bbR} \xi(\lambda) t e^{-t\lambda} \; d\lambda.
\end{equation}
Thus with mild regularity assumptions on $\xi,$ relatively isospectral complete metrics have $\xi \equiv 0,$ i.e., are isophasal.

As mentioned above, the metrics considered by Borthwick-Perry
are relatively isospectral, as are those considered by Hassell-Zelditch,
if one allows the compact sets $K_i$ to have boundary.
In the case of Borthwick-Perry this is a consequence of a Poisson formula
\cite[Theorem 2.1]{Borthwick-Perry}, due in this context to Borthwick,
that shows that the resonance set determines the (renormalized) trace of the
wave kernel, and hence the (renormalized) trace of the heat kernel.
Similarly, in \cite[(1.6)]{Hassell-Zelditch}, it is pointed out that the scattering phase determines
a renormalized trace of the heat kernel.
It is straightforward to write the relative trace of the heat kernels as the difference of their renormalized traces.


Let us now introduce the type of Riemannian metrics considered here.
Let $\overline{M}$ be a smooth compact surface with boundary $\pa \overline{M}$ with $n$ marked points $p_1,\ldots, p_n\in \overline{M}\setminus\pa \overline{M}$.  We suppose $\pa \overline{M}$ is the disjoint union of two types of boundaries, $\pa_F\overline{M}$ and $\pa_b\overline{M}$, each of which being a finite union of circles.  Consider then the possibly non-compact surface $M= \overline{M}\setminus(\pa_F\overline{M}\cup\{p_1,\ldots,p_n\})$.

\begin{definition}
Let $M$ be as described above. A \textbf{Funnel-cusp-boundary metric} ($\Fcb$-metric for short)
is a Riemannian metric $g$ on $M$ such that:
\begin{itemize}
\item[(i)] there exist a collar neighborhood
$$
   c_F: \pa_F\overline{M}\times [0,\nu)_x\to \overline{M},
$$
a Riemannian metric $h_F$ on $\pa_F\overline{M}$ and  a smooth function $\varphi_F\in \CI(\pa_F\overline{M}\times [0,\nu)_x)$  locally constant on $\pa_F\overline{M}\times \{0\}$  such that
$$
      c_F^*g= e^{\varphi_F}\left(\frac{dx^2+ \pr_F^*h_F}{x^2}\right),
$$
where $\pr_{F}$ is the natural projection onto $\pa_F\overline{M}$;
\item[(ii)] For each marked point $p_i$,  there exist a neighborhood $\cV_i\subset \overline{M}$, coordinates $u,v$ on $\cV_i$ with $u(p_i)=v(p_i)=0$, and a function $\varphi_i\in\CI(\overline{M})$
such that near $p_i$,
$$
        g= e^{\varphi_i}\left( \frac{du^2+dv^2}{r^2(\log r)^2} \right),
$$
where $r= \sqrt{u^2+v^2}$.
\end{itemize}
In other words, the metric is asymptotically hyperbolic near $\pa_F\overline{M}$ while it is conformal to a cusp near each marked point $p_i$.  We say $(M,g)$ is a $\Fcb$-Riemannian surface.
\label{Fcb.1}\end{definition}

\begin{figure}[h]
\centerline{\includegraphics[height=5cm]{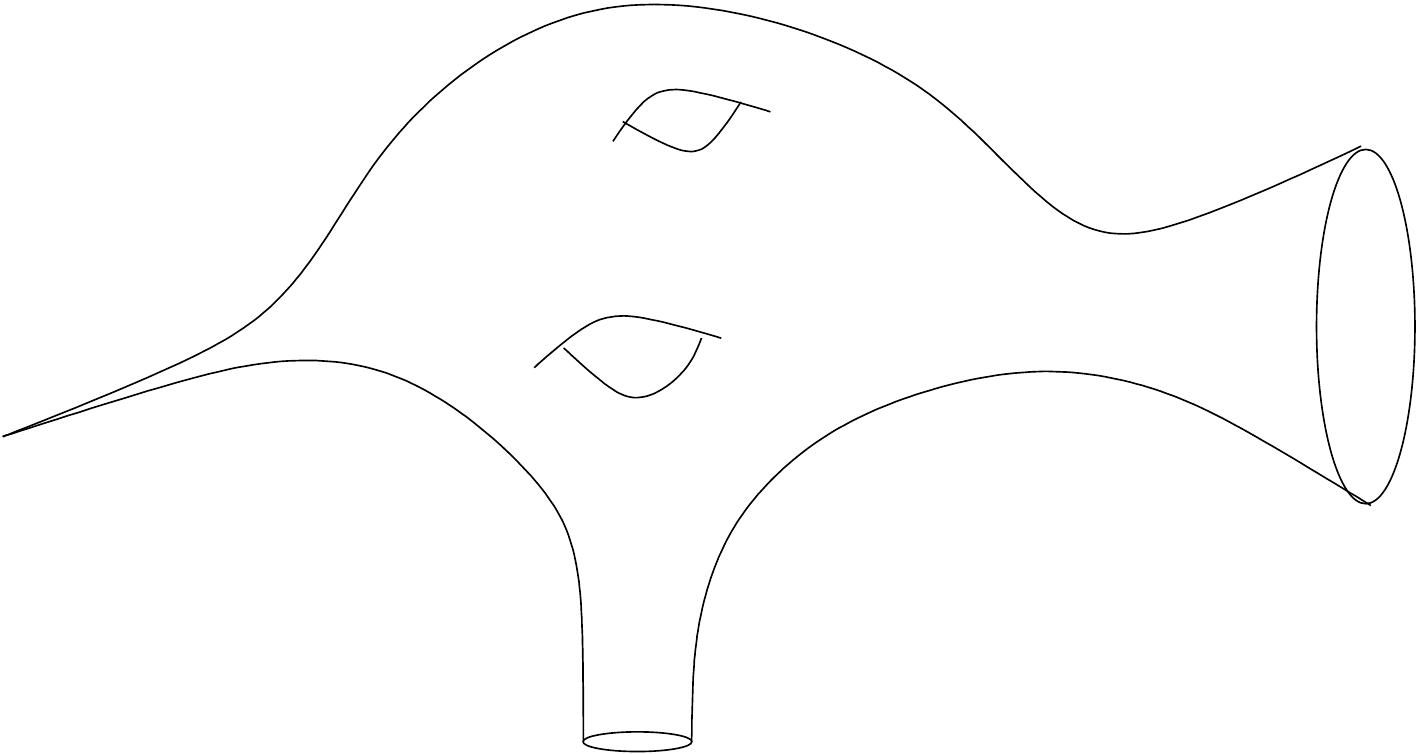}}
\caption{$(M,g)$ a $\Fcb$-surface} \label{fig1}
\end{figure}

For such a $\Fcb$-metric $g$ on $M$, we will consider the corresponding (positive) Laplacian $\Delta_g$ with Dirichlet boundary condition on $\pa M:= \pa_b\overline{M}$.  When $\pa M=\emptyset$, $\Fcb$-metrics correspond to some of the $F-\hc$ metrics considered in \cite{AAR}.  If $\pa_F\overline{M}=\emptyset$ and $\overline{M}$ has no marked points, the spectrum of $\Delta_g$ is discrete.  Otherwise, each boundary component $Y$ of $\pa_F\overline{M}$  gives rise to a band of continuous spectrum $[\frac{e^{-c_Y}}{4},\infty)$ (of infinite multiplicity), where the constant $c_Y$ is the restriction of $\varphi_F$ 
to $Y$, while each marked point $p_i$ gives rise to a band of continuous spectrum $[\frac{e^{-\varphi_i(p_i)}}{4}, \infty)$ (of multiplicity one). In particular, the continuous spectrum is bounded below by a positive constant.  Notice also that $0$ is in the spectrum if and only if $\pa_F\overline{M}=\pa_b\overline{M}=\emptyset$.

The main result of this paper is to establish compactness for sets of relatively isospectral $\Fcb$-Riemannian surfaces.  More precisely, we prove the following theorem.
\begin{theorem}\label{thm:RelIso}
Let $(M_i,g_i)$ be a sequence of $\Fcb$-Riemannian surfaces, isospectral relative to $\cU_{\infty}$.
Then there is a Riemannian surface $(M, g_{\infty}),$ a subsequence $(M_{i_k},g_{i_k}),$ and a sequence of diffeomorphisms
\begin{equation*}
    \phi_k: M \lra M_{i_k},
    \Mwith \phi_{\ell}\circ \phi_{\ell'}^{-1}\rest{\cU_{\infty}} = \Id \text{ for any } \ell, \ell'
\end{equation*}
such that the metrics $\phi_k^*g_{i_k}$ converge to $g_{\infty}$ in $\CI.$
\end{theorem}

The proof of Theorem \ref{thm:RelIso} consists in reducing to the case treated by Borthwick and Perry (or by Osgood, Phillips and Sarnak if there are only cusps) via a conformal surgery.  Namely, using the fact that a cusp is conformal to a punctured disk, we conformally modify the metrics near each cusp to obtain punctured disks that we fill in.  Similarly, near each boundary component of $M$, we modify the metric conformally from an incomplete cylinder to a complete funnel hyperbolic near infinity.  Doing another conformal transformation, we can assume that these metrics are also hyperbolic at the other funnel ends.  Under these conformal surgeries, the new metrics will generally no longer be relatively isospectral.  However, since each metric undergoes the same conformal transformation in a region where all the metrics were the same, the relative local heat invariants stay the same.  From the local nature of Polyakov's formula for the variation of the determinant, one can also hope that the relative determinant will remain unchanged.  Using a finite speed propagation argument, we are able to show that this is indeed the case -- even though for a conformal surgery at a point, the deformation does not have uniformly bounded curvature nor global Sobolev inequalities. This reduces the problem to the situations treated in \cite{OPS2, Borthwick-Perry}, from which the compactness of the set of relatively isospectral $\Fcb$-metrics follows.

\begin{acknowledgements}
The authors are grateful to David Borthwick, Gilles Carron, Andrew Hassell, Rafe Mazzeo and Richard Melrose for helpful conversations. 
\end{acknowledgements}

\paperbody

\section{Conformal surgeries}\label{fpsa.0}

Let $(M,g)$ be a $\Fcb$-Riemannian surface.  
  Fix a marked point $p\in \overline{M}\setminus \pa \overline{M}$ and choose coordinates $u$ and $v$ in a neighborhood $\cV$ of $p$ such that $u(p)=v(p)=0$ and such that in these coordinates the metric takes the form
$$
               g= e^{f}\frac{du^2+dv^2}{r^2(\log r)^2}=e^{f}\frac{dr^2 + r^2 d\theta^2}{r^2(\log r)^2}= e^f\left(\frac{d\rho^2}{\rho^2}+ \rho^2 d\theta^2\right),  \quad f\in\CI(\cV),
$$
where we are using the polar coordinates $u=r\cos \theta$, $v=r\sin\theta$ and $\rho= \frac{-1}{\log r}$.Without loss of generality, we can assume that the neighborhood $\cV$ is given by the open disk of radius $r=\frac45$.  Let $\sigma\in \CI_c(\cV)$ be a function taking values between $0$ and $1$ with $\sigma\equiv 1$ for $r<\frac14$ and $\sigma\equiv 0$ for $r>\frac12$.  Consider then the function $\psi(\eps,r)$ given by
\begin{equation}
  \psi(\epsilon,r)= \sigma(r)\left( \frac{r^2(\log r)^2}{\epsilon^2+(\epsilon^2+r^2)(\log \sqrt{r^2+\epsilon^2})^2} \right) +(1-\sigma(r))
\label{fpsa.1}\end{equation}
Extending $\psi_{\eps}(r)=\psi(\eps,r)$ by 1 outside the neighborhood $\cV$ allows us to consider the family of metrics
\[
      g_{\eps}= \psi_{\eps}g.
\]
For $\epsilon>0$, this metric is smooth at the marked point $p$.  However,
as $\eps$ approaches $0$, the point $p$ is pushed to infinity so that in the limit we recover the metric $g$.
We say the metric $g$ undergoes a \textbf{conformal surgery at the marked point $p$}.

\begin{figure}[h]
\setlength{\unitlength}{1.5cm}
\begin{picture}(9,3)
\thicklines

\qbezier(0,1)(2,1)(2,1.5)
\qbezier(0,1)(2,1)(2,0.5)
\put(0.7,0){$\epsilon=0$}

\qbezier(3.5,1.2)(5,1.2)(5,1.5)
\qbezier(3.5,0.8)(5,0.8)(5,0.5)
\qbezier(3.5,0.8)(3,0.8)(3,1)
\qbezier(3.5,1.2)(3,1.2)(3,1)
\put(3,1){\circle*{0.1}}
\put(2.7,1){$p$}
\put(3.7,0){$\epsilon=\frac12$}

\qbezier(7.2,1.5)(6.7,1)(7.2,0.5)
\put(6.95,1){\circle*{0.1}}
\put(6.7,1){$p$}
\put(7,0){$\epsilon=1$}

\thinlines
\qbezier(2,1.5)(1.9,1)(2,0.5)
\qbezier(2,1.5)(2.1,1)(2,0.5)

\qbezier(5,1.5)(4.9,1)(5,0.5)
\qbezier(5,1.5)(5.1,1)(5,0.5)

\qbezier(7.2,1.5)(7.1,1)(7.2,0.5)
\qbezier(7.2,1.5)(7.3,1)(7.2,0.5)

\end{picture}
\caption{Conformal surgery at a point p}\label{csap}
\end{figure}
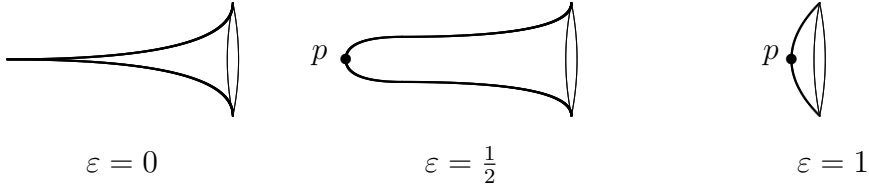

Similarly, suppose now that $\pa M$ is nonempty and fix a boundary component $\pa_i M$.  Instead of a conformal surgery at $p_i$,
we can consider one at $\pa_i M$.  Thus, we take now $\cV\cong \pa_i M\times [0,1)_r$ to be a collar neighborhood of  $\pa_i M$ in $M$
and let $\theta$ be the angular variable on $\pa_i M\cong \bbS^1$.  Near the boundary, the metric is of the form
$$
       g=  e^{f}(dr^2 + d\theta^2),   \quad f\in\CI(\cV).
$$

In $\cV$, consider a function $\psi(\eps,r)$ such that
\begin{equation}
\begin{gathered}
    (\eps^2+r^2) \psi(\eps,r) \in \CI(\cV ),\\
     \psi(\eps, r) = 1 \Mif \eps >1/2 \Mor r > 1/2,\\
    (\eps^2+r^2) \psi(\eps, r) = e^{w} \Mif \eps^2+r^2 \approx 0,
\end{gathered}
\label{fpsa.1b}\end{equation}
where $w\in \CI(\cV)$ is such that $w+f$ is constant on $\pa_iM$.  For instance, we can take $w=-f$.
Extending $\psi_{\eps}(r)=\psi(\eps,r)$ by 1 outside the neighborhood $\cV$ allows us to consider the family of metrics
\[
      g_{\eps}= \psi_{\eps}g.
\]
In this case, the boundary component $\pa_i M$ is pushed to infinity as $\epsilon$ approaches zero, so that in the limit the metric becomes asymptotically hyperbolic in that end,
$$
   g_0=     e^{w+f}\frac{dr^2+d\theta^2}{r^2}, \quad r \ \mbox{small.}
$$
If we pick $w=-f$, then $g_0$ is in fact hyperbolic near infinity.
We think of this family of metrics as a \textbf{conformal surgery at $\pa_i M$}.

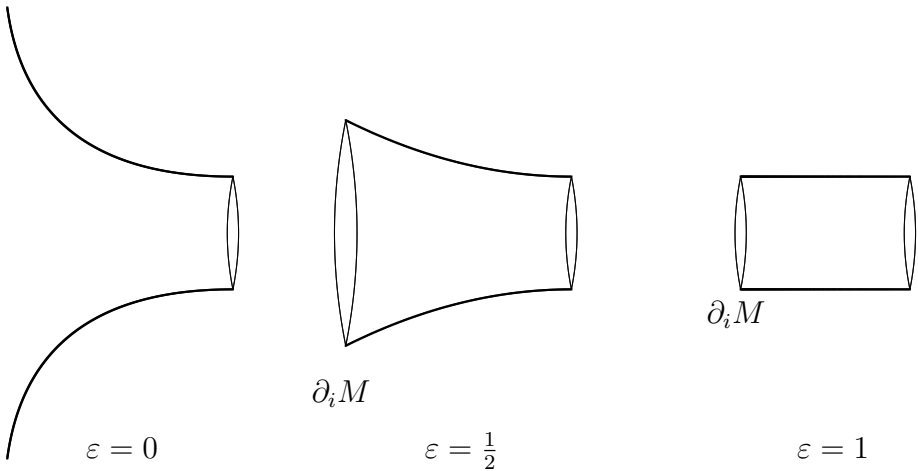
\begin{figure}[h]
\setlength{\unitlength}{1.5cm}
\begin{picture}(9,4)
\thicklines

\qbezier(0,4)(0.2,2.5)(2,2.5)
\qbezier(0,0)(0.2,1.5)(2,1.5)
\put(0.7,0){$\epsilon=0$}

\qbezier(3,3)(4,2.5)(5,2.5)
\qbezier(3,1)(4,1.5)(5,1.5)
\put(2.7,0.5){$\pa_i M$}
\put(3.7,0){$\epsilon=\frac12$}

\qbezier(6.5,1.5)(7,1.5)(8,1.5)
\qbezier(6.5,2.5)(7,2.5)(8,2.5)
\put(6.2,1.2){$\pa_i M$}
\put(7,0){$\epsilon=1$}

\thinlines
\qbezier(2,1.5)(1.9,2)(2,2.5)
\qbezier(2,1.5)(2.1,2)(2,2.5)

\qbezier(5,1.5)(4.9,2)(5,2.5)
\qbezier(5,1.5)(5.1,2)(5,2.5)
\qbezier(3,1)(2.8,2)(3,3)
\qbezier(3,1)(3.2,2)(3,3)

\qbezier(8,1.5)(7.9,2)(8,2.5)
\qbezier(8,1.5)(8.1,2)(8,2.5)
\qbezier(6.5,1.5)(6.4,2)(6.5,2.5)
\qbezier(6.5,1.5)(6.6,2)(6.5,2.5)

\end{picture}
\caption{Conformal surgery at a boundary component $\pa_i M$}\label{csabc}
\end{figure}

To study the determinant, it will be important for us to know that the spectra stay away from zero under conformal surgeries at a point
or at a boundary component.

\begin{theorem}
Let $(M,g)$ be a $\Fcb$-Riemannian surface.   Let the family $g_{\epsilon}$ be  a conformal surgery at point $p$ or at boundary component $\pa_i M$.  Then the smallest non-zero eigenvalue of $\Delta_{g_{\epsilon}}$ is bounded below by a constant $c>0$ independent of $\epsilon$.
\label{csp.1}\end{theorem}
\begin{proof}
Suppose first that $g_{\epsilon}$ is a conformal surgery at a point $p$.
 Let $\lambda_{\epsilon}>0$ be the smallest positive eigenvalue of $\Delta_{g_{\epsilon}}$ and let $u\in \CI(M)\cap L^2(M,g_{\epsilon})$ be the corresponding eigenfunction, so that $\Delta_{g_{\epsilon}}u=\lambda_{\epsilon} u$.  When $\epsilon>0$, notice that $u\in \CI(\overline{M})$,  so its restriction to $M\setminus\{p\}$ will be in $L^2(M,g_0)$, in fact, in the domain of $\Delta_{g_0}$.  Now, we can find a constant $C>0$ such that
$$
                \dvol_{g_{\epsilon}}\le C \dvol_{g_0},  \quad \forall \, \epsilon \ge0.
$$
On the other hand, since our change of metric is conformal, $\| du \|^2_{L^2(M,g_{\epsilon})}= \| du \|^2_{L^2(M,g_{0})}$.   Thus, if zero is not in the spectrum, that is, if $(M,g)$ has at least one funnel end or if it has a non-empty boundary, we have that
$$
  \lambda_{\epsilon} =   \frac{\| du \|^2_{L^2(M,g_{\epsilon})}}{\| u \|^2_{L^2(M,g_{\epsilon})}} \ge
  \frac{1}{C} \frac{\| du \|^2_{L^2(M,g_{0})}}{\| u \|^2_{L^2(M,g_{0})}}\ge \frac{\lambda_0}{C},
$$
where $\lambda_0= \inf\Spec(\Delta_0)$.  It therefore suffices to take $c=   \frac{\lambda_0}{C}$ in this case.  If instead $(M,g)$ has an empty boundary and no funnel end, we know that zero is an eigenvalue.  Thus, in this case, we have instead
$$
 \lambda_{\epsilon}=  \inf_{\Pi} \max_{v\in\Pi} \frac{\| dv \|^2_{L^2(M,g_{\epsilon})}}{\| v \|^2_{L^2(M,g_{\epsilon})}},
$$
where $\Pi$ runs over 2-dimensional subspaces of $\CI_c(M)$.  Since $\CI_c(M)$ is densely contained in the domain of $\Delta_{g_0}$, we have again
$$
 \lambda_{\epsilon}=  \inf_{\Pi} \max_{v\in\Pi} \frac{\| dv \|^2_{L^2(M,g_{\epsilon})}}{\| v \|^2_{L^2(M,g_{\epsilon})}} \ge    \frac{1}{C} \inf_{\Pi}\max_{v\in\Pi} \frac{\| dv \|^2_{L^2(M,g_{0})}}{\| v \|^2_{L^2(M,g_{0})}}\ge \frac{\lambda_0}{C},
 $$
 where $\lambda_0= \inf (\Spec(\Delta_0)\setminus \{0\})$,
so that we can still take $c=\frac{\lambda_0}{C}$ to obtain the result.

If instead $g_{\epsilon}$ is a conformal surgery at a boundary component $\pa_i M$, notice that  if $u\in\CI(M)$ is an eigenfunction of $\Delta_{g_{\epsilon}}$ for $\epsilon>0$, then, since $u$ is zero on $\pa_i M$, it is in $L^2(M,g_0)$, in fact in the domain of $\Delta_{g_0}$.  On the other hand, we can clearly find a constant $C>0$ such that
$$
      \dvol_{g_{\epsilon}}\le C \dvol_{g_0}  \quad \forall \epsilon\in [0,1].
$$
With these two facts, we can proceed as before to obtain the result.
\end{proof}

\section{$L^2$-estimates for the heat kernel}

In \cite{CGT1982}, some $L^2$ estimates are obtained for the heat kernel using a finite speed propagation argument.  More precisely, they obtained a bound on the norm of the heat kernel acting on $L^2$-functions.  As observed by Donnelly in \cite{Donnelly1987}, using the Sobolev embedding, this also gives an estimate for the Hilbert-Schmidt norm of  the heat kernel.  Since we will use this observation several times, we will, for the convenience of the reader, go through this argument in detail.

\begin{lemma}
Let $(M,g)$ be a $\Fcb$-Riemannian surface.  Let $\cU$ and $G$ be two open sets in $M$ and let $d=d_{g}(\cU,G)\ge 0$ be the distance between them.  Suppose the closure of $G$ is compact in $M\setminus\pa M$.  If $d>0$ there exists a constant $C_G$ depending on $G$ such that 
$$
     \int_{\cU}  |e^{-t\Delta_g}(x,y)e^{-t\Delta_g}(x,y')| dx  \le C_G e^{\frac{-d^2}{8t}} \quad \forall \; y,y'\in G, \; \forall \; t>0.
$$
If instead $d=0$ then, given $\nu>0,$ there exists a constant $C_{G,\nu}>0$ such that
$$
\int_{\cU}  |e^{-t\Delta_g}(x,y)e^{-t\Delta_g}(x,y')| dx  \le C_{G,\nu} \quad \forall \; y,y'\in G, \; \forall \; t\ge\nu.$$  
\label{don.1}\end{lemma}  
\begin{proof}
Suppose first that $d>0$.  Let $\tG$ be an open set relatively compact in $M$ such that $\overline{G}\subset \tG$  with $\pa\tG$ smooth and $\td=d_g(\cU,\tG)> \frac{d}{\sqrt{2}}$.  Let $\chi\in \CI_c(\tG)$ be a nonnegative cut-off function with $\chi\equiv 1$ in a neighborhood of $G$.  On $\tG\times \tG$, consider the distribution
\begin{equation}
   W(y,y') := \int_\cU e^{-t\Delta_g}(x,y)\chi(y) e^{-t\Delta_g}(x,y') \chi(y') dx.
\label{don.2}\end{equation}  
The reason for inserting the cut-off function $\chi$ in the definition of $W$ is to be able to integrate by parts later on.  

To prove the lemma   we will first show that the distributions $\Delta^k_{\tG,y} \Delta^\ell_{\tG,y'}W(y,y')$ are in $L^2,$ where $k, \ell\in\bbN$ and $\Delta_{\tG}$ is the Laplacian of the metric $g$ on $\tG$ with Dirichlet boundary conditions.  Let $\{u_i\}$, $i\in \bbN$, be an orthonormal basis of $L^2(\tG,g)$ given by eigenfunctions of $\Delta_{\tG}$ and let $0<\lambda_1\le \lambda_2\le \cdots$ be the corresponding eigenvalues.  Then $\{u_i(y) u_j(y')\}$, $i,j\in\bbN$ is an orthonormal basis of $L^2(\tG\times \tG, g\times g)$.  On $\tG$, we know from the Weyl law that there exists a constant $c>0$ such that $\lambda_j \sim cj$ as $j$ tends to infinity, so that
$$
     K:= \sqrt{ \sum_{i,j} \frac{1}{\lambda_i^2 \lambda_j^2}} <\infty.  
$$   
On the other hand, by the $L^2$-estimate of \cite[Corollary~1.2]{CGT1982} applied to the heat kernel, we know that, for any $k, \ell \in \bbN,$ there exists a constant $C_{k,\ell}>0$ such that
\begin{equation}
 \left| \int_{\tG\times \tG} (\Delta_{\tG,y}^k \Delta_{\tG,y'}^\ell W(y,y')) u_i(y)u_j(y') dy dy'    \right| \le C_{k,\ell} e^{-\frac{\td^2}{4t}}< C_{k,\ell} e^{-\frac{d^2}{8t}}, 
\label{don.3}\end{equation}
for all $i,j \in \bbN,$ and $t>0.$
Now, given $v\in L^2(\tG\times \tG, g\times g)$, it can be written as 
$$
   v(y,y')= \sum_{i,j} \mu_{ij} u_i(y)u_j(y'),  \quad \mbox{with} \; \|v\|^2_{L^2} = \sum_{i,j} |\mu_{ij}|^2.
$$
Using \eqref{don.3}, we can pair it with $\Delta^k_{\tG,y}\Delta_{\tG,y'}^{\ell}W(y,y')$, namely,
\begin{equation}
\begin{aligned}
&\left| \int_{\tG\times \tG}(\Delta_{\tG,y}^k \Delta_{\tG,y'}^\ell W(y,y')) v(y,y') dydy'  \right| = 
\left| \int_{\tG\times \tG}(\Delta_{\tG,y}^{k+1} \Delta_{\tG,y'}^{\ell+1} W(y,y'))\Delta_{\tG,y}^{-1} \Delta_{\tG,y'}^{-1} v(y,y') dydy'  \right| \\
 &\hspace{5cm}= \left| \sum_{i,j} \frac{\mu_{ij}}{\lambda_i \lambda_j} \int_{\tG\times \tG}  (\Delta_{\tG,y}^{k+1} \Delta_{\tG,y'}^{\ell+1} W(y,y')) u_i(y)u_j(y') dy dy'  \right| \\
 &\hspace{5cm}\le \left( \sum \frac{1}{\lambda_i^2\lambda_j^2}\right)^{\frac12} \left( \sum_{i,j} |\mu_{ij} |^2\right)^{\frac12} C_{k+1,\ell+1} e^{-\frac{d^2}{8t}} \\
 &\hspace{5cm}= K \| v\|_{L^2} C_{k+1,\ell+1} e^{-\frac{d^2}{8t}}, \quad \forall \; t>0.
\end{aligned}
\label{don.4}\end{equation}
Since $v$ is arbitrary, this means by Riesz theorem that $\Delta^k_{\tG,y}\Delta^\ell_{\tG,y'} W$ is in $L^2$ with 
$$
      \| \Delta^k_{\tG,y} \Delta^\ell_{\tG,y'} W\|_{L^2}\le K C_{k+1,\ell+1} e^{-\frac{d^2}{8t}}.
$$
We can thus apply the Sobolev embedding theorem to bound the $C^0$-norm of $W$ on $G\times G$, giving the desired estimate.  Notice also that we have shown $W$ is smooth on $G\times G$. 
 
If instead $d=0$, we can follow the same argument, except that instead of \eqref{don.3}, there is a constant $C_{k,\ell,\nu}$ depending on $\nu>0$ such that 
$$
\left| \int_{\tG\times \tG} (\Delta_{\tG,y}^k \Delta_{\tG,y'}^\ell W(y,y')) u_i(y)u_j(y') dy dy'    \right| \le C_{k,\ell,\nu}, \quad \forall \; i,j\in \bbN, \quad \forall\; t\ge\nu.
$$

\end{proof}

For some of the applications of this estimate, we will allow the open set $G$ to move towards infinity and it will be essential to understand how the constant $C_G$ grows under such a change.  From the proof Lemma~\ref{don.1}, this constant depends in a subtle way on the eigenvalues of the Laplacian $\Delta_{\tG}$.  To control the growth of $C_G$, we will instead derive the estimate, using the fact (established in the proof of Lemma~\ref{don.1}) that $W$ is smooth.  This requires some preparation.

Consider the hyperbolic metric
$$
     g_{\bbH^2}=  \frac{dx^2+dy^2}{y^2}
$$
on the upper-half plane $\bbH^2= \{ z=x+iy\in \bbC \; ; \; y>0\}$.  The hyperbolic cusp metric (also called the horn metric) is obtained from this metric by taking the quotient of $\bbH^2$ by the isometric action of $\bbZ$ generated by $z\mapsto z+1$.  By making the change of coordinates $u=\log y$, we can write the hyperbolic cusp metric as
$$
      g_{\horn} = du^2 + e^{-2u}dx^2 \;\; \mbox{on}  \; H= \bbR\times (\bbR/\bbZ).
$$
With respect to the change of variable $x=\theta$ and $\rho=e^{-u}$, we see this simply corresponds to the metric $\frac{d\rho^2}{\rho^2}+ \rho^2d\theta^2$ considered in the conformal surgery at a point.  However, for the study of the constant $C_G$, it will be more convenient to work with the coordinates $(u,x)$.  For $a>0$, consider the open set
$$
\cU_{a}= \{ (u,x)\in \bbR\times (\bbR/\bbZ)\; : \; u<a\}.
$$

\begin{lemma}
Let $W\in\CI(H\times H)$ be such that there exist constants $C_{a,k,l}$ depending on $a, k$ and $l$ such that
$$
    \left|  \int_{\cU_a\times \cU_a} W(y,y') \Delta^k_{g_{\hc}}u(y) \Delta^l_{g_{\hc}}v(y') dydy' \right| \le C_{a,k,l} \|u\|_{L^2(\cU_a,g_{\hc})}  \|v\|_{L^2(\cU_a,g_{\hc})}
$$
for all $u,v\in L^2(\cU_a, g_{\hc})$ with supports compactly included in $\cU_a$.  Then there exists a constant $C$ independent of $a$, $k$, and $l$ such that
$$
  \sup_{\cU_a\times \cU_a} |W| \le C (e^{4a}+1) \left( \sum_{k,l=0}^3 C_{a+1,k,l}\right)
$$
\label{del.1}\end{lemma}

\begin{proof}
Notice first that for $y_0,y_0'$ in $\cU_a$, we have that 
\begin{equation}
   W(y_0, y_0') = \int_{\cU_a\times \cU_a} W(y,y') \delta_{y_0}(y) \delta_{y_0'}(y')dydy',
\label{del.2}\end{equation}
where $\delta_{y_0}$ is the Dirac delta function centered at $y_0$.  Let $\psi \in C(\bbH^2)$ be a cut-off function taking values between $0$ and $1$ such that 
$$
         \psi(u,x)= \left\{ \begin{array}{ll}
                             1, &  0\le x \le \frac14 \;\mbox{and} \;  u\le 0 , \\
                             0, &  |x|\ge \frac12 \; \mbox{or} \; u \ge 1.
           \end{array} \right.
$$
Then define the translated  function $\psi_{a,b}(u,x)= \psi(u-a,x-b)$.  Clearly, there is a constant $K$ independent of $a$ and $b$ such that
\begin{equation}   
  \sup |\Delta_{\bbH^2} \psi_{a,b}| \le K (e^{2a}+1),  \quad \sup  |\nabla \psi_{a,b} |_{g_{\bbH^2}} \le K (e^a+1),  \quad  \sup |\psi|=1.
\label{del.4}\end{equation}
Let $q: \bbH^2\to \bbH^2/\bbZ$ be the quotient map.  Given $y_0\in \cU_a$, choose $\ty_0\in \bbH^2$ such that $q(\ty_0)=y_0$ and let $u_{\ty_0}$ be the unique solution in the Sobolev space $L^2_1(\bbH^2, g_{\bbH^2})$ to
$$
        \delta_{\ty_0}= \Delta^3_{\bbH^2} u_{\ty_0}.
$$ 
By symmetry, there is a constant $K_1$ independent of $\ty_0\in \bbH^2$ such that
\begin{equation}
        \| u_{\ty_0} \|_{L^2_1(\bbH^2,g_{\bbH^2})} \le K_1.
\label{del.5}\end{equation}
Take $b:=x(\ty_0)$ to be the $x$ coordinate of $\ty_0$  so that $\psi_{a,b}\equiv 1$ in a neighborhood of $\ty_0$.  Using this cut-off function, this means the equality
\begin{equation}
\delta_{\ty_0} = \sum_{j=0}^3 \Delta^j_{\bbH^2} \alpha_j, 
\label{del.3}\end{equation}
descends to an equality on $\cU_{a+1}$, where
\begin{equation}
\begin{aligned}
\alpha_0 &= -(\Delta_{\bbH^2}\psi_{a,b}) \Delta^2_{\bbH^2} u_{\ty_0} + 2\nabla\psi_{a,b}\cdot \nabla(\Delta_{\bbH^2}u_{\ty_0}) \\
 \alpha_1 &= - (\Delta_{\bbH^2} \psi_{a,b}) \Delta_{\bbH^2} u_{\ty_0} - \psi_{a,b} \Delta^2_{\bbH^2}u_{\ty_0}+ 2\nabla \psi\cdot \nabla (\Delta_{\bbH^2} u_{\ty_0}) \\
 \alpha_2 &= -(\Delta_{\bbH^2}\psi_{a,b}) u_{\ty_0} + 2\nabla\psi_{a,b}\cdot \nabla u_{\ty_0} \\
 \alpha_3 &= \psi_{a,b} u_{\ty_0}.
\end{aligned}
\label{del.3b}\end{equation}
By \eqref{del.5},  $\|\alpha_3\|_{L^2(\bbH^2,g_{\bbH^2})}\le K_1$.  On the other hand, since $\nabla\psi_{a,b}$ is supported away of $\ty_0$, notice by elliptic regularity and using \eqref{del.4} and \eqref{del.5} that there is a positive constant $K_2$ independent of $a$ and $\ty_0$ such that 
\begin{equation}
        \|\alpha_j\|_{L^2(\bbH^2,g_{\bbH^2})}\le K_2(e^{2a}+1), \quad j=0,1,2.           
\label{del.3c}\end{equation}
Thus, plugging  \eqref{del.3} in \eqref{del.2} and using the hypothesis of the lemma, we obtain
$$
     \sup_{\cU_a\times \cU_a} |W| \le C(e^{4a}+1)(\sum_{k,l=0}^3 C_{a+1,k,l})
$$
for a constant $C$ depending only on $K$, $K_1$ and $K_2$.     

\end{proof}

\section{Long time behavior of the relative trace}

As mentioned in the introduction, the crux of the proof of compactness of relatively isospectral metrics is showing that the relative determinant of the Laplacian and the relative heat invariants are unchanged by conformal surgery. The Laplacian is conformally covariant in dimension two, so this will follow from knowing that the variation of these invariants is continuous, so ultimately from showing that the relative trace of the heat kernel is uniformly continuous along a conformal surgery. This section provides one of the ingredients for such a result, namely, a good uniform control of the relative trace for large time.  \\

Fix a family of functions $\psi_{\epsilon}$ as in \eqref{fpsa.1} or \eqref{fpsa.1b} with support in an open set $\cV$ containing the point or the boundary component at which the conformal surgery is performed.  Let $g$ and $h$ be two $\Fcb$-metrics on M such that  $h=g$ on $\cU\cup \cV$, where $\cU= M\setminus K$ and $K$ is a compact set.  For these metrics, we can consider the corresponding conformal surgeries  $g_\eps= \psi_\eps g$ and   $h_\eps= \psi_\eps h$.  Clearly, for all $\eps>0$, we have that $h_\eps =g_\eps$ on $\cU\cup \cV$, while for $\eps=0$, we have that $h_0=g_0$ on $(\cU\cup \cV)\setminus \{p\}$ for a conformal surgery at a point $p$,  and $h_0=g_0$ on $(\cU\cup \cV)\setminus \pa_i M$ for a conformal surgery at a boundary component $\pa_i M$.
By the results of Bunke \cite{Bunke1992} and Carron \cite{Carron:DetRel}, we know the difference of heat kernels
\begin{equation}
      e^{-t\Delta_{g_\eps}} - e^{-t\Delta_{h_\eps}}
\label{fpsa.2}\end{equation}
is trace-class for all $\eps\ge 0$ and all $t>0$.  We will in fact need a more precise statement about the trace norm of this difference of heat kernels.
\begin{proposition}
Given $T>\nu>0$, there exists a positive constant $C$ such that
$$
\| e^{-t\Delta_{g_\eps}} - e^{-t\Delta_{h_\eps}}\|_{\Tr}\le C\quad \forall \epsilon\in [0,1], \quad \forall \ t\in [\nu,T].
$$
\label{utn.1}\end{proposition}
\begin{proof}

For metrics that coincide outside a compact set, Bunke \cite{Bunke1992} established a bound on the trace norm of the relative heat kernel.  By partially following his proof, but also adding a new twist, we will show that this bound can be achieved uniformly in $\eps.$

Let $S$ be the point or the boundary component where the conformal surgery is taking place.
Let us choose two open sets in $M,$ $\cW_1$ and $\cW_2,$ such that
\begin{itemize}
\item[(i)] $M= \cW_1 \cup \cW_2$;
\item[(ii)] $\overline{\cW}_1\subset \cV\cup \cU$ and $\overline{\cW}_2$ is compact;
\item[(iii)] $\left.g_\epsilon\right|_{\cW_1}=\left. h_\epsilon \right|_{\cW_1}$ for all $\epsilon\in [0,1]$;
\item[(iv)] $\left.g_\epsilon\right|_{\cW_2}=\left. g_0 \right|_{\cW_2}$ and
  $\left.h_\epsilon\right|_{\cW_2}=\left. h_0 \right|_{\cW_2}$ for all $\epsilon\in [0,1]$.
\end{itemize}
We will also choose $\cW_1$ and $\cW_2$ so that they are contained in bigger open sets $\wt\cW_1$ and $\wt\cW_2$ satisfying the same properties and such that
$$
        g_\epsilon=h_\epsilon=g=h \quad \mbox{on}
        \; \wt\cW_1\setminus \cW_1 \; \mbox{and}\; \wt\cW_2\setminus \cW_2, \; \forall \ \epsilon\in [0,1].
$$
In \cite{Bunke1992}, the strategy is to estimate the trace-norm of the difference of the heat kernels by considering the restriction of the difference to $\cW_1\times \cW_1$, $\cW_2\times \cW_2$,
$\cW_1\times \cW_2$ and $\cW_2\times \cW_1,$
and then applying finite propagation speed estimates.

Let us consider first the region $\cW_2\times \cW_2.$
Let $\phi\in \CI_c(M)$ be a function with $\phi\equiv 1$ near $\cW_2$ and $\supp\phi \subset \wt\cW_2$.  Let $\chi\in \CI_c(M)$ be another function such that $\chi\equiv 1$ on $\supp \phi$ and $\supp \chi \subset \wt\cW_2$.  Finally, let $\gamma\in \CI(M)$ be a function with $\gamma\equiv 1$ on $\supp(1-\phi)$ and $\gamma\equiv 0$ on $\cW_2$.  Consider then the approximate heat kernel
$$
     H_{\epsilon}(t,x,y)= \gamma(x) e^{-t\Delta_{g_0}}(x,y)(1-\phi(y))+
     \chi(x)e^{-t\Delta_{g_\epsilon}}(x,y)\phi(y).
$$
Since $\lim_{t\to 0} H_{\epsilon}(t,x,y)= \delta_{x,y}$, we obtain via Duhamel's principle that
$$
e^{-t\Delta_{g_0}}(x,y) - H_\epsilon(t,x,y)=
  - \int_0^t \int_{M\setminus S}
  e^{-s\Delta_{g_{0}}}(x,z) \left(\frac{\pa}{\pa t} +\Delta_{g_{0}}\right) H_\epsilon(t-s,z,y) dzds.
$$
Thus, for $x$ and $y$ in $\cW_2$, we have that $H_\epsilon(t,x,y)= e^{-t\Delta_{g_\epsilon}}$, so that
$$
e^{-t\Delta_{g_0}}(x,y) - e^{-t\Delta_{g_\epsilon}}(x,y)=
  - \int_0^t \int_{G}
  e^{-s\Delta_{g_{0}}}(x,z)  E_\epsilon(t-s,z,y) dzds,
$$
where $G\subset \wt\cW_2\setminus \cW_2$ is the support of $d\chi$ and
$$
   E_\epsilon(t,z,y)= (\Delta_{g}\chi)(z)e^{-t\Delta_{g_{\epsilon}}}(z,y) -2
   \langle \nabla_z\chi(z), \nabla_z e^{-t\Delta_{g_\epsilon}}(z,y)\rangle_g
$$
If $P_G$ and $P_{\cW_2}$ are the projection operators obtained by multiplying by the characteristic functions of $G$ and $\cW_2$, this can be rewritten as
\begin{equation}
P_{\cW_2}( e^{-t\Delta_{g_0}} -e^{-t\Delta_{g_\epsilon}})P_{\cW_2}=
-\int_0^t (P_{\cW_2} e^{-s\Delta_{g_0}}P_G)(P_G E_\epsilon(t-s)P_{\cW_2})ds.
\label{utn.3}\end{equation}
If $d= \dist_{g}(G,\cW_2)$ is the distance between $G$ and $\cW_2$ with respect to the metric $g$, then using the $L^2$ estimates of \cite{CGT1982}, we know by Lemma~\ref{don.1} (\cf \cite[p.69]{Bunke1992}) that there exists a positive constant $C$ depending on $G$  such that
$$
 \| P_{\cW_2} e^{-s\Delta_{g_0}}P_G\|_{\HS}\le Ce^{-\frac{d^2}{8s}}, \quad
 \| P_{G} E_\epsilon(t-s)P_{\cW_2}\|_{\HS}\le Ce^{-\frac{d^2}{8(t-s)}},
$$
where $\|\cdot\|_{\HS}$ is the Hilbert-Schmidt norm.  Since we have $\|AB\|_{\Tr}\le \|A\|_{\HS}\|B\|_{\HS}$ for two Hilbert-Schmidt operators $A$ and $B$, we see from \eqref{utn.3} that
$$
     \| P_{\cW_2}( e^{-t\Delta_{g_0}} -e^{-t\Delta_{g_\epsilon}})P_{\cW_2}  \|_{\Tr}\le C
$$
for a positive constant $C$ independent of $t\in(0,T]$ and $\epsilon\in [0,1]$.  We obtain similarly that
$$
\| P_{\cW_2}( e^{-t\Delta_{h_0}} -e^{-t\Delta_{h_\epsilon}})P_{\cW_2}  \|_{\Tr}\le C.$$
Combining these two inequalities, this means that for $t\in [\nu,T]$,
$$
\| P_{\cW_2}( e^{-t\Delta_{g_\epsilon}} -e^{-t\Delta_{h_\epsilon}})P_{\cW_2}  \|_{\Tr}\le 2C+ \max_{t\in[\nu,T]} \| P_{\cW_2}( e^{-t\Delta_{g_0}} -e^{-t\Delta_{h_0}})P_{\cW_2}  \|_{\Tr},$$
giving the desired uniform bound in $\epsilon$ in that region.

For the region $\cW_1\times \cW_1$, Bunke writes the difference of heat kernels as a sum of products of Hilbert-Schmidt operators
whose norm is bounded by $Ce^{-\frac{d^2}{8t}}$ using a finite speed propagation argument, see \cite[Theorem~3.4]{Bunke1992}.
Here, the positive constants $C$ and $d$ depend on a choice (independent of $\epsilon$) of compact region $G$ in
$\wt\cW_1\setminus \cW_1$ and on the metric $h=h_\epsilon= g_\epsilon=g$ in that region.  In particular, the constants $C$ and $d$ can
be chosen to be the same for all $\epsilon\in [0,1]$.

For the regions $\cW_1\times \cW_2$, we cannot proceed as in \cite{Bunke1992}, since we do not have a uniform bound on the curvature in $\epsilon$ for a conformal surgery at a point.  Instead, we will use a finite speed propagation argument.  Since we already control the trace norm on $\cW_1\times \cW_1$ and $\cW_2\times \cW_2$,  it suffices to control the trace norm in a smaller open set $\hW_1\times \hW_2\subset \cW_1\times \cW_2$, where the open sets $\hW_1$ and $\hW_2$ can be chosen to be disjoint and such that 
$$
    M\setminus \cW_2 \subset \hW_1\subset \cW_1 \quad\mbox{and}\quad M\setminus \cW_1 \subset \hW_2\subset \cW_2,
$$
with $g_{\epsilon}=h_{\epsilon}=g_0=h_0$ on $\cW_1\setminus \hW_1$ and $\cW_2\setminus \hW_2$.
 
Let $\phi_1, \phi_2\in \CI(M)$ be nonnegative functions such that 
\begin{itemize}
\item[(i)]  $\phi_i\equiv 1$ in an open neighborhood of $\hW_i$ for $i=1,2$;  
\item[(ii)] $\supp \phi_1 \cap \supp \phi_2 = \emptyset$. 
\end{itemize}
Since the kernel 
$$
  E_{\epsilon}(t,x,y)= \phi_1(x) (e^{-t\Delta_{g_{\epsilon}}}(x,y) -e^{-t\Delta_{h_{\epsilon}}}(x,y))\phi_2(y)
$$ 
is supported away from the diagonal, its limit as $t\to 0$ is zero.   Thus, by Duhamel's principle, we have that 
$$
  E_{\epsilon}(t,x,y)= \int_0^t \int_{M\setminus S}  e^{-s\Delta_{g_{\epsilon}}}(x,z) \left(\frac{\pa}{\pa t} + \Delta_{g_{\epsilon}}\right) E_{\epsilon}(t-s,z,y) dzds.
$$
Using the fact $g_{\epsilon}=h_{\epsilon}$ on $\cW_1$ and $\supp \phi_1 \subset \cW_1$, we see that this can be rewritten as 
$$
 E_{\epsilon}(t,x,y)= \int_0^t\int_G e^{-s\Delta_{g_{\epsilon}}}(x,z) G_{\epsilon}(t-s,x,y) dz ds, 
$$
where $G\subset \cW_1\setminus \hW_1$ is the support of $d\phi_1$ and 
\begin{multline}
  G_{\epsilon}(t,z,y)= (\Delta_g \phi_1)(z) (e^{-t\Delta_{g_{\epsilon}}}(z,y)-e^{-t\Delta_{h_{\epsilon}}}(z,y)) \\ + 2\langle \nabla_z \phi_1(z), \nabla_z (e^{-t\Delta_{h_{\epsilon}}}(z,y)- e^{-t\Delta_{g_{\epsilon}}}(z,y))\rangle_g.
\end{multline}
If $P_{\hW_1}$, $P_{\hW_2}$ and $P_{G}$ are the projection operators obtained by multiplying by the characteristic functions of $\hW_1$,  $\hW_2$ and $G$, then this can be reformulated as follows,
$$
  P_{\hW_1}(e^{-t\Delta_{g_{\epsilon}}}- e^{-t\Delta_{h_{\epsilon}}})P_{\hW_2}= \int_0^t (P_{\hW_1}e^{-s\Delta_{g_{\epsilon}}} P_G) (P_G G_{\epsilon}(t-s) P_{\hW_2}) ds.
$$
Then, for $d:= \min_{i=1,2} \dist_g(G,\hW_i) >0$, we see from Lemma~\ref{don.1} that there is a  positive constant C depending on $G$ such that 
$$
   \| P_{\hW_1}e^{-s\Delta_{g_{\epsilon}}} P_G\|_{\HS}\le C e^{-\frac{d^2}{8s}}, \quad
   \|P_G G_{\epsilon}(t-s) P_{\hW_2} \|_{\HS} \le C e^{-\frac{d^2}{8(t-s)}}.
$$
We can thus conclude as before that there is a constant $C_1>0$ such that  
\begin{equation}
  \|P_{\hW_1}(e^{-t\Delta_{g_{\epsilon}}}- e^{-t\Delta_{h_{\epsilon}}})P_{\hW_2}\|_{\Tr}\le C_1, \quad \forall t\in [0,T], \; \forall \epsilon \in [0,1]. 
\label{lde.1}\end{equation}
 
Finally, for the region $\cW_2\times \cW_1$, we also only need to control the trace norm on
$\hW_2\times \hW_1$.   Since $P_{\hW_2}(e^{-t\Delta_{g_{\epsilon}}}- e^{-t\Delta_{h_{\epsilon}}})P_{\hW_1}$ is the adjoint\footnote{When acting on half-densities} of $ P_{\hW_1}(e^{-t\Delta_{g_{\epsilon}}}- e^{-t\Delta_{h_{\epsilon}}})P_{\hW_2}$, the desired estimate follows from \eqref{lde.1} in this case.   
\end{proof}

Using Theorem~\ref{csp.1} and Proposition~\ref{utn.1}, we obtain the following estimate for the behavior of the relative trace as $t$ tends to infinity.

\begin{corollary}
Let $\mu>0$ be a uniform lower bound for the positive spectrum of $\Delta_{g_{\epsilon}}$ and $\Delta_{h_{\epsilon}}$ for all $\epsilon\in [0,1]$.  Then there exist $T>0$ and $K>0$ independent of $\epsilon$ such that
\[
   \| e^{-t\Delta_{g_{\epsilon}}}-e^{-t\Delta_{h_{\epsilon}}}\|_{\Tr}  \le K e^{-\frac{\mu}2 t}, \quad \forall t\ge T.
\]
\label{ed.1}\end{corollary}
\begin{proof}
Recall that any $\Fcb$ metric has a punctured neighborhood of $0$ disjoint from its spectrum.
Assume first that $0$ is not in the spectrum.  Since there is a constant $C\ge 1$ such that $\frac{g_\epsilon}{C}\le h_\epsilon\le C g_\epsilon$ for all $\epsilon$, we know  by the spectral theorem that for $t_0= \frac{\log (2C^2)}{\mu}$, we have
$$
       \| e^{-t_0 \Delta_{g_{\epsilon}}}\| \le \frac12 , \quad
       \| e^{-t_0\Delta_{h_{\epsilon}}}\| \le \frac12,
$$
where $\| \cdot \|$ is the operator norm defined with respect to the norm of $L^2(M,g_\epsilon)$.  For $t\ge 2t_0$, notice that
\begin{equation}
\begin{aligned}
 \| e^{-t\Delta_{g_{\epsilon}}}-e^{-t\Delta_{h_{\epsilon}}}\|_{\Tr} &= \| e^{-\frac{t}{2}\Delta_{g_{\epsilon}}}(e^{-\frac{t}{2}\Delta_{g_{\epsilon}}}- e^{-\frac{t}{2}\Delta_{h_{\epsilon}}}) +
    (e^{-\frac{t}{2}\Delta_{g_{\epsilon}}}-  e^{-\frac{t}{2}\Delta_{h_{\epsilon}}})e^{-\frac{t}{2}\Delta_{h_{\epsilon}}}\|_{\Tr}  \\
    & \le (\| e^{-\frac{t}{2}\Delta_{g_{\epsilon}}} \|  +\| e^{-\frac{t}{2}\Delta_{h_{\epsilon}}} \|  )  \|  e^{-\frac{t}{2}\Delta_{g_{\epsilon}}} -e^{-\frac{t}{2}\Delta_{h_{\epsilon}}}  \|_{\Tr} \\
    &\le  \|  e^{-\frac{t}{2}\Delta_{g_{\epsilon}}} -e^{-\frac{t}{2}\Delta_{h_{\epsilon}}}  \|_{\Tr}.\end{aligned}
\label{ed.2}\end{equation}
Applying this inequality finitely many times, we see that for all $t\ge 2t_0$,
$$
\| e^{-t\Delta_{g_{\epsilon}}}-e^{-t\Delta_{h_{\epsilon}}}\|_{\Tr} \le \max_{\tau\in [t_0,2t_0]} \| e^{-\tau\Delta_{g_{\epsilon}}}-e^{-\tau\Delta_{h_{\epsilon}}}\|_{\Tr}.$$

By Proposition~\ref{utn.1}, this means the relative trace is uniformly bounded for $t\ge 2t_0$ and $\epsilon\in [0,1]$.  Now, using the fact that $\| e^{-t\Delta_{g_{\epsilon}}}\|  \le e^{-t\mu}$ and $\| e^{-t\Delta_{h_{\epsilon}}}\|  \le c e^{-t\mu}$ for some constant $c>0$ only depending on $g$ and $h$ and proceeding as in \eqref{ed.2}, we thus have for $t\ge 2t_0$,
\begin{equation}
\begin{aligned}
\| e^{-t\Delta_{g_{\epsilon}}}-e^{-t\Delta_{h_{\epsilon}}}\|_{\Tr} &\le
(\| e^{-\frac{t}{2}\Delta_{g_{\epsilon}}} \|  +\| e^{-\frac{t}{2}\Delta_{h_{\epsilon}}} \|  )  \|  e^{-\frac{t}{2}\Delta_{g_{\epsilon}}} -e^{-\frac{t}{2}\Delta_{h_{\epsilon}}}  \|_{\Tr} \\
&\le (1+c)e^{-\frac{\mu}{2}t } \max_{\tau\in [t_0,2t_0]} \| e^{-\tau\Delta_{g_{\epsilon}}}-e^{-\tau\Delta_{h_{\epsilon}}}\|_{\Tr}.
\end{aligned} \label{ed.3}\end{equation}
By Proposition~\ref{utn.1},  $ \max_{\tau\in [t_0,2t_0]} \| e^{-\tau\Delta_{g_{\epsilon}}}-e^{-\tau\Delta_{h_{\epsilon}}}\|_{\Tr}$ is bounded above by a positive constant independent of $\epsilon$.  This gives the desired result.  If zero is in the  spectrum, we can obtain the same result by first projecting off the constants.
\end{proof}

\section{Finite time behavior of the relative trace}

To introduce and study the relative determinant, we need to obtain some good control on the relative trace as $\epsilon$ tends to $0$. We will adapt the methods used in \cite{Donnelly1987} to show that  the relative heat trace is continuous in $\eps$ for small $t.$

\begin{theorem}
For $T>0$,  the functional
\[
    \eps\mapsto \Tr(e^{-t\Delta_{g_\eps}} - e^{-t\Delta_{h_\eps}})
\]
is continuous at $\eps =0$ uniformly with respect to $t\in (0,T]$.
\label{fpsa.3}\end{theorem}

Intuitively, we expect Theorem~\ref{fpsa.3} to hold  from the fact the singular behavior of $g_\eps$ and $h_{\eps}$ should cancel out.  What allows us to turn these local considerations into a statement about the heat kernels is a finite propagation speed argument.  We will proceed in three steps.  We will use the notation $S$ to denote either $\{p\}$ or $\pa_i M$ depending on whether we are considering a conformal surgery at a point or at boundary component.

\begin{lemma}
There exist constants $d>0$ and $C>0$ such that
$$
   | \Tr( e^{-t\Delta_{g_{\epsilon}}} - e^{-t\Delta_{h_{\epsilon}}}) -\Tr( e^{-t\Delta_{g_{0}}} - e^{-t\Delta_{h_{0}}}) | < C t e^{-\frac{d^2}{8t}} $$
for all $t>0$ and $\epsilon\in [0,1]$.
\label{sta.1}\end{lemma}
\begin{proof}
Note that we expect such a rapid decay as $t$ tends to zero from the fact that the short time asymptotics of these heat kernels cancel out.  The precise estimate will be obtained via a finite speed propagation argument.

Let $\cW_1$, $\wt\cW_1$, $\cW_2$ and $\wt\cW_2$ be open sets as in the proof of Proposition~\ref{utn.1}.
To estimate the difference of relative traces on $\cW_1$, let $\phi\in \CI(M)$ be a function with $\phi\equiv 1$ near $\cW_1$ and $\supp \phi \subset \wt\cW_1$.  Let $\chi\in \CI(M)$ be another function with $\chi\equiv 1$ on $\supp \phi$ and $\supp\chi\subset \wt\cW_1$.  Let also $\gamma\in \CI_c(M)$ be a function with $\gamma\equiv 1$ on $\supp(1-\phi)$.  Consider then the approximate heat kernel
$$
  E(t,x,y)= \gamma(x) e^{-t\Delta_{g_{\epsilon}}}(x,y) (1-\phi(y))+
  \chi(x) e^{-t\Delta_{h_{\epsilon}}}(x,y)\phi(y).
$$

By Duhamel's principle, we have that
\begin{equation}
  e^{-t\Delta_{g_{\epsilon}}}(x,y) - E(t,x,y)= \\
  - \int_0^t \int_{M\setminus S}
  e^{-s\Delta_{g_{\epsilon}}}(x,z) \left(\frac{\pa}{\pa t} +\Delta_{g_{\epsilon}}\right) E(t-s,z,y) dzds.
\label{sta.2}\end{equation}
If we assume now that $x$ and $y$ are equal and lie in $\cW_1$, then $E_{\epsilon}(t,x,x)= e^{-t\Delta_{h_{\epsilon}}}(x,x)$, so proceeding as in \cite[Proposition~5.1]{Donnelly1987}, we have that
\begin{multline}
\int_{\cW_1\setminus S} | e^{-t\Delta_{g_{\eps}}}(x,x)- e^{-t\Delta_{h_{\eps}}}(x,x)| dx \le \\
C \int_0^t \int_G \left( \int_{\cW_1\setminus S} | e^{-s\Delta_{g_{\eps}}}(x,z)|^2 dx\right)^{\frac12} \\ 
\left( \left( \int_{\cW_1\setminus S} | \nabla_ze^{-(t-s)\Delta_{h_{\eps}}}(z,x)|^2 dx\right)^{\frac12} +  
\left( \int_{\cW_1\setminus S} |e^{-(t-s)\Delta_{h_{\eps}}}(z,x)|^{2} dx\right)^{\frac12} \right)dz ds,
\label{sta.3}\end{multline} where $C>0$ is a constant depending on the norm of $d\chi$ and $\Delta_h\chi$ and $G\subset \wt\cW_1\setminus \cW_1$ is a compact set containing the support of $d\chi$.  By Lemma~\ref{don.1}, we have
\begin{equation}
\begin{gathered}
\int_{\cW_1\setminus S} | e^{-s\Delta_{g_{\eps}}}(x,z)|^2 dx \le C_1 e^{-\frac{d^2}{8s}}, \quad  \int_{\cW_1\setminus S} 
|e^{-(t-s)\Delta_{h_{\eps}}}(z,x)|^2 dx \le C_1 e^{-\frac{d^2}{8(t-s)}},  \\
\int_{\cW_1\setminus S} | \nabla_ze^{-(t-s)\Delta_{h_{\eps}}}(z,x)|^2 dx \le C_1 e^{-\frac{d^2}{8(t-s)}},
\end{gathered}
\label{sta.4}\end{equation}
where $C_1$ is a constant depending on $G$, so is independent of $\epsilon$, while $d>0$ is chosen to be smaller than the distance between $\cW_1$ and $\supp d\chi$ with respect to the metric $h$ (recall that $h=h_{\epsilon}=g_{\epsilon}=g$ on $\wt\cW_1\setminus \cW_1$). Combining \eqref{sta.3} and \eqref{sta.4}, we obtain 
\begin{equation}
\int_{\cW_1\setminus S} | e^{-t\Delta_{g_\eps}}(x,x) - e^{-t\Delta_{h_\eps}}(x,x)|dx < {2}\Vol(G)C C_1 t e^{-\frac{d^2}{8t}} \quad \forall \ \epsilon\in [0,1], \forall \ t>0.
\label{sta.5}\end{equation}

Using the fact that $g_\epsilon=g_0$ and $h_{\epsilon}=h_0$ on $\wt\cW_2$ and that $g_\epsilon=h_\epsilon=g=h$ on $\wt\cW_2\setminus \cW_2$, we can proceed in a similar way to obtain that
\begin{equation}
\begin{gathered}
\int_{\cW_2} | e^{-t\Delta_{g_\eps}}(x,x) - e^{-t\Delta_{g_0}}(x,x)|dx < C t e^{-\frac{d^2}{8t}}, \\
\int_{\cW_2} | e^{-t\Delta_{h_\eps}}(x,x) - e^{-t\Delta_{h_0}}(x,x)|dx < C t e^{-\frac{d^2}{8t}},
\end{gathered}
\label{sta.6}\end{equation}
for all $\epsilon\in [0,1]$  and  $ t>0$, for potentially different constants $C>0$ and $d>0.$
Finally, combining \eqref{sta.5} and \eqref{sta.6} gives the result.

\end{proof}
An easy consequence of Lemma~\ref{sta.1} is the following.
\begin{corollary}
The relative heat invariants in the asymptotic expansion
\begin{equation}\label{RelHeatInv}
    \Tr( e^{-t\Delta_{g_{\epsilon}}} - e^{-t\Delta_{h_{\epsilon}}} )
    \sim
    t^{-1} \sum_{k\geq 0} a_k(g_{\epsilon},h_{\epsilon}) t^k.
\end{equation}
are preserved under conformal surgery, i.e. $a_{k}(g_\eps,h_\eps)=a_{k}(g_0,h_0)$ for all $k$ and all $\epsilon$.
\label{relhin.1}\end{corollary}

\begin{lemma}
Given $\delta>0$ and $T>0$, there exist an open set $A \subset \cV$ containing $S$ and $\eps_0>0$ such that for all $t<T$ and all $\eps\in [0,\eps_0)$,
\[
     \int_{A\setminus S} | e^{-t\Delta_{g_\eps}}(x,x) - e^{-t\Delta_{h_\eps}}(x,x)|dx <\delta.
\]
\label{fpsa.4}\end{lemma}

\begin{proof}
This is a finite propagation speed argument as in \cite[Proposition~5.1]{Donnelly1987}, namely we consider the approximate heat kernel
\begin{equation}
  E(t,x,y)= \gamma(x) e^{-t\Delta_{g_{\eps}}}(x,y)(1-\phi(y))+ \chi(x) e^{-t\Delta_{h_{\eps}}}(x,y)\phi(y),
\label{fpsa.5}\end{equation}
where $\phi,\chi$ and $\gamma$ are smooth functions on $M$ with $\phi\equiv 1$ near $p$ and $\phi\equiv 0$ when $r>\frac{1}{2}$ on $\cV$ (and more generally outside of $\cV$), $\chi \equiv 1$ when $r\le\frac12$  and $\chi\equiv 0$ when $r>\frac34$, and $\gamma\equiv 1$ on $\supp (1-\phi)$ and $\supp \gamma$ is disjoint from $p$.  Let $G$ be a compact set containing the support of $d\chi$ and choose $A\subset \cV$ such that $\phi\equiv 1$ on $A$.  By choosing $A$ and $\eps_0>0$ sufficiently small, we can insure
$d_{g_{\eps}}(A,G)>d$ when $\eps\in [0,\eps_0]$,  where $d$ is a large positive number to be chosen later.  Using Duhamel's principle as in \cite[Proposition~5.1]{Donnelly1987}, we have
\begin{multline}
\int_{A\setminus S} | e^{-t\Delta_{g_{\eps}}}(x,x)- e^{-t\Delta_{h_{\eps}}}(x,x)| dx \le \\
C \int_0^t \int_G \left( \int_{A\setminus S} | e^{-s\Delta_{g_{\eps}}}(x,z)|^2 dx\right)^{\frac12}\\ 
\left( \left( \int_{A\setminus S} | \nabla_ze^{-(t-s)\Delta_{h_{\eps}}}(z,x)|^2 dx \right)^{\frac12} +
\left( \int_{A\setminus S} |e^{-(t-s)\Delta_{h_{\eps}}}(z,x)|^2 dx\right)^{\frac12} \right) dz ds,
\label{fpas.6}\end{multline}
where $C$ is a constant depending on $G$.  By Lemma~\ref{don.1} and using the fact $g_{\epsilon}=g$ and $h_{\epsilon}=h$ near $G$, we have
\begin{equation}
\begin{gathered}
\int_{A\setminus S} | e^{-t\Delta_{g_{\eps}}}(x,z)|^2 dx \le C_1 e^{-\frac{d^2}{8t}}, \quad \int_{A\setminus S} |e^{-t\Delta_{h_{\eps}}}(z,x)|^2 dx \le C_1 e^{-\frac{d^2}{8t}}, \\ \int_{A\setminus S} | \nabla_ze^{-t\Delta_{h_{\eps}}}(z,x)|^2 dx \le C_1 e^{-\frac{d^2}{8t}},
\end{gathered}
\label{fpas.7}\end{equation}
where the positive constant $C_1$ only depends on $G$ and is thus independent of $\eps$.  By taking $A$ and $\eps_0$ sufficiently small, we can make $d$ as large as we want.  From \eqref{fpas.6}, given $\delta>0$ and $T>0$, we can thus choose $A$ and $\eps_0$ so that
\[
     \int_{A\setminus S} | e^{-t\Delta_{g_\eps}}(x,x) - e^{-t\Delta_{h_\eps}}(x,x)|dx <\delta,
\]
for all $\eps\in [0,\eps_0]$ and $t\in (0,T]$.
\end{proof}
We need also to control the trace on the complement of $A$.

\begin{lemma}
Let $N\subset M$ be an open set with $N\cap S= \emptyset$.  Then given $\delta>0$ and $T>\nu>0$, there exists $\eps_0>0$ such that for all $\eps\in [0,\eps_0]$ and $t\in [\nu,T]$,
\[
   \int_N | e^{-t\Delta_{g_{\eps}}}(x,x) - e^{-t\Delta_{g_0}}(x,x)|dx <\delta,
   \;
   \int_N | e^{-t\Delta_{h_{\eps}}}(x,x) - e^{-t\Delta_{h_0}}(x,x)|dx <\delta.
\]
\label{fpas.8}\end{lemma}
\begin{proof}
We will prove the lemma for the metric $g_{\eps}$, the proof being the same for the metric $h_{\eps}$.  Without loss of generality, by taking $N$ bigger if needed, we can assume $M\setminus N$ is contained in the neighborhood $\cV$.  This time, we consider $\phi\in \CI(M)$ with $\phi\equiv 1$ on $N$ and $\phi\equiv 0$ near $S$, $\chi\in \CI(M)$ with the same properties and such that $\chi\equiv 1$ on the support of $\phi$, and we choose a function $\gamma\in \CI_c(\overline{\cV})$ with $\gamma\equiv 1$ on $\supp(1-\phi)$.  With these functions, we define the approximate heat kernel
\begin{equation}
E(t,x,y)= \gamma(x) e^{-t\Delta_{g_{0}}}(x,y)(1-\phi(y))+
  \chi(x)e^{-t\Delta_{g_{\eps}}}(x,y)\phi(y).
\label{fpas.9}\end{equation}
Using Duhamel's principle, we have
\begin{equation}
  e^{-t\Delta_{g_0}}(x,y) - E(t,x,y)=
  - \int_0^t \int_{M\setminus S}
  e^{-s\Delta_{g_0}}(x,z) \left(\frac{\pa}{\pa t} +\Delta_{g_0}\right) E(t-s,z,y) dzds.
\label{fpas.10}\end{equation}
Suppose now that $x$ and $y$ are equal and lie in $N$.  Then $E(t,x,x)= e^{-t\Delta_{g_{\eps}}}(x,x)$. Writing $g_{\eps}= e^{\varphi_\eps}g_0$, we also have,
\begin{equation}
\begin{aligned}
\left( \frac{\pa}{\pa t} + \Delta_{g_0} \right) e^{-t\Delta_{g_{\eps}}} &=
  \left( \frac{\pa}{\pa t} + \Delta_{g_\eps} + (e^{\varphi_{\eps}}-1)\Delta_{g_{\eps}} \right) e^{-t\Delta_{g_{\eps}}}\\
  &=
  (e^{\varphi_{\eps}}-1)\Delta_{g_{\eps}} e^{-t\Delta_{g_{\eps}}}.
\end{aligned}
\label{fpas.11}\end{equation}
Thus, from \eqref{fpas.10},  when $x$ and $y$ are equal and lie in $N$, we have
\begin{equation}
\begin{aligned}
e^{-t\Delta_{g_0}}(x,x) - e^{-t\Delta_{g_{\eps}}}(x,x) &=
2\int_{0}^t \int_G e^{-s\Delta_{g_0}}(x,z) \langle \nabla_z\chi, \nabla_z e^{-(t-s)\Delta_{g_{\eps}}}(z,x) \rangle dzds \\
&-\int_{0}^t \int_G e^{-s\Delta_{g_0}}(x,z) (\Delta_{g_0}\chi(z)) e^{-(t-s)\Delta_{g_{\eps}}}(z,x) dzds  \\
&+
\int_0^t \int_{G'} e^{-s\Delta_{g_0}}(x,z) (1-e^{\varphi_\eps}) \chi(z) \Delta_{g_{\eps}} e^{-(t-s)\Delta_{g_{\eps}}} (z,x) dzds,
\end{aligned}\label{fpas.12}\end{equation}
where $G = \supp d\chi \subset M\setminus S$ is a compact set 
and $G'\subset \cV\setminus S$ is a compact set containing $\supp \chi \cap \supp (1-e^{\varphi_\eps})$.

When we integrate with respect to $x$ on $N$, the first two terms on the right hand side of \eqref{fpas.12} can be bounded as before by
\begin{multline*}
C\int_0^t \int_G \left( \int_N |e^{-s\Delta_{g_0}}(x,z)|^2 dx\right)^{\frac12}\cdot \\
  \left( \left( \int_N |e^{-(t-s)\Delta_{g_{\eps}}}(z,x)|^2 dx\right)^{\frac12} + \left( \int_N |\nabla_z e^{-(t-s)\Delta_{g_{\eps}}}(z,x)|^2 dx \right)^{\frac12} \right) dzds,
\end{multline*}
where $C$ is a constant depending on the norm of $d\chi$ and $\Delta_{g_0}\chi$ on $G$ with respect to the norm of $g_{\eps}$.  By Lemma~\ref{don.1},
we can  bound this expression
by
\begin{equation}\widetilde{C}\sqrt{\Vol(G,g_0)\Vol(G,g_\epsilon)}\sqrt{C_{G,g_0}C_{G,g_{\epsilon}}} e^{-\frac{d_0^2+d_\eps^2}{16t}}
\label{fpas.13}\end{equation} where $d_{\eps}$ is the distance between $G$ and $N$ with respect to the metric $g_{\eps}$,
 $\widetilde{C}$ is a constant depending on $C$ and $\nu$, and $C_{G,g_{\epsilon}}$ is the optimal constant for the estimate of Lemma~\ref{don.1} for $g_{\epsilon}$ on $G$.

 In the case of a surgery at a boundary component, using the fact $g_0$ is quasi-isometric to a $\Fcb$-metric which is hyperbolic near infinity, we see using Lemma~\ref{del.1} with the constant $a$ fixed that  we can take $C_{G,g_0}$ to be independent of $G$.  On the other hand, for fixed $G$, we can take the constant $C_{G,g_\epsilon}$ to be as close as we want to $C_{G,g_0}$ by taking $\eps$ sufficiently small.    Now, at the cost
 of changing $\chi$ and taking $\eps$ sufficiently small, we can make $d_0$ and $d_\eps$ as large as we want.
Choosing $\chi$ to be independent of $\theta$ near $p$, this can be achieved while keeping the norm of $d\chi$ bounded by a fixed constant $K$ with respect to $g_0$, so that its norm with respect to $g_{\eps}$ will be bounded by $K+1$ if $\eps$ is small enough. During such a procedure, the volume of $G$ with respect to $g_0$  satisfies an estimate of the form $\Vol(G,g_0)\le C_0e^{d_0}$ for some fixed constant $C_0$, and again, for fixed $G$, by taking $\epsilon$ small enough, we can make $\Vol(G,g_{\epsilon})$ as close as we want to $\Vol(G,g_0)$.  
This means that  by choosing $\chi$ and $\eps$ suitably, we can make \eqref{fpas.13} as small as we want.

In the case of a conformal surgery at a point, we need to use the fact that $g_0$ is quasi-isometric to a hyperbolic cusp metric near the point $p$, so that by Lemma~\ref{del.1}, the Sobolev constant of $g_0$ on $G$ satisfies 
an estimate of the form
$$
        C_{G,g_0}\le K (e^{4d_0}+1)
$$
for some constant $K$.    
Thus, by choosing $G$ sufficiently far away from $N$, we can make the term $C_{G,g_0}e^{-\frac{d_0^2}{8t}}$
as small as we want.  Since a cusp end has finite area, this can be done in such a way   that the volume of $G$ with respect to $g_{\epsilon}$  is bounded above by a constant independent of $d_0$ and $\epsilon$.

Since on the other hand we can, for fixed $G$, make $C_{G,g_{\epsilon}}$ arbitrarily close to $C_{G,g_0}$ by taking $\epsilon$ 
sufficiently small, we see we can again make \eqref{fpas.13} as small as we want by taking $G$ sufficiently far away from $N$ and 
$\epsilon$ sufficiently small.

For the third term in \eqref{fpas.12}, we note that its integral with respect to $x$ on $N$ can be bounded by
\begin{equation}
C \int_0^t  \int_{G'} \left( \int_N | e^{-s\Delta_{g_0}}(x,z)|^2 dx\right)^{\frac12}
 \left( \int_N |\Delta_{g_{\eps}} e^{-(t-s)\Delta_{g_{\eps}}}(z,x)|^2 dx\right)^{\frac12} (1-e^{\varphi_{\eps}(z)}) dz ds.
\label{fpsa.14}\end{equation}
Since we are assuming $t\ge\nu>0$, each integral in $x$ can be uniformly bounded (for $G'$ fixed) using Lemma~\ref{don.1} with $d=0$.  Thanks to the term $(1-e^{\varphi_{\eps}})$, the overall expression can be made arbitrarily small by taking $\eps>0$ sufficiently small.
\end{proof}
These three lemmas can then be combined to give the proof of Theorem~\ref{fpsa.3}.
\begin{proof}[Proof of Theorem~\ref{fpsa.3}]
Given $T>0$ and $\delta>0$, we need to find $\eps_0>0$ such that when $\eps\in [0,\eps_0]$,
\[
\int_{M\setminus S} | (e^{-t\Delta_{g_\eps}}(x,x) - e^{-t\Delta_{h_\eps}}(x,x))-
(e^{-t\Delta_{g_0}}(x,x) - e^{-t\Delta_{h_0}}(x,x))|dx <\delta \quad \forall \ t\in (0,T].
\]
By Lemma~\ref{sta.1}, we can find $\nu>0$ such that this integral is smaller than $\frac{\delta}{3}$ for $t\le \nu$.
By Lemma~\ref{fpsa.4}, we can find $\eps_0>0$ and an open set $A\subset \cV$ containing $S$ such that the integral restricted to $A\setminus S$ is smaller than $\frac{\delta}{3}$ when $\eps\in [0,\eps_0]$ and $t\in (0,T]$.  On the other hand, choosing an open set $N\subset (M\setminus S)$ containing the complement of  $A$, we know by Lemma~\ref{fpas.8} that by taking $\eps_0$ smaller if needed, we can insure the integral restricted to $N$ is also bounded by $\frac{\delta}{3}$ for $\eps\in[0,\eps_0]$ and $t\in [\nu,T]$, from which the result follows.
\end{proof}

\section{The relative determinant}
By   Corollary~\ref{ed.1} and Corollary~\ref{relhin.1}, the relative zeta function given by
$$ \zeta(\Delta_{g_\epsilon}, \Delta_{h_{\epsilon}},s)= \frac{1}{\Gamma(s)} \int_0^{\infty}  t^{s-1}\Tr(e^{-t\Delta_{g_\eps}} - e^{-t\Delta_{h_\eps}}) dt
$$
 is well defined for $\Re s>1$. In fact, using the short-time asymptotic expansion \eqref{RelHeatInv}, the relative zeta function can be extended meromorphically to $s\in \bbC$ with at worst simple poles in $s$, but with $s=0$ a regular point.
Thus, a relative determinant can be defined by
\begin{equation}
    \det(\Delta_{g_{\epsilon}}, \Delta_{h_{\epsilon}})= \exp \left( - \zeta'(\Delta_{g_\epsilon}, \Delta_{h_{\epsilon}},0)  \right).
\label{RelDet}\end{equation}
\begin{lemma}
For $\epsilon>0$, we have
$$
  \frac{d}{d\epsilon}\det(\Delta_{g_{\epsilon}}, \Delta_{h_{\epsilon}})     =  0.
$$
\label{const.1}\end{lemma}
\begin{proof}
Using the regularized trace as in \cite{AAR}, we can write the relative trace as a difference of two regularized traces.  Similarly, we can write the relative determinant as a quotient of two regularized determinants.  Applying the Polyakov formula of \cite{AAR} to this ratio of regularized determinants, we see that the contribution coming from one regularized determinant is canceled by the other, from which the result follows.
\end{proof}
\begin{remark}
In \cite{AAR} the surfaces considered have no boundary, but since the discussion about the regularized trace and the Polyakov formula is local near the cusps and the funnels,  the extension to the boundary case is automatic.  As in \cite{OPS1}, one simply needs to add an extra term in the Polyakov formula of \cite{AAR} expressed in terms of the geodesic curvature of the boundary.
\end{remark}

This gives immediately the following.
\begin{theorem}
For the families of metrics $g_{\epsilon}$ and $h_{\epsilon}$, the relative determinant $\det(\Delta_{g_{\epsilon}}, \Delta_{h_{\epsilon}})$ is independent of $\epsilon$.
\label{csp.3}\end{theorem}
\begin{proof}
By Corollary~\ref{ed.1}, Lemma~\ref{sta.1} and Theorem~\ref{fpsa.3}, the relative determinant is a continuous function of $\epsilon$.  On the other hand, by Lemma~\ref{const.1}, it is constant for $\epsilon>0$.  By continuity, it is therefore constant for $\epsilon\ge 0$.
\end{proof}

We can use this to obtain a similar result for conformal deformations near $\pa_F\overline{M}$.
\begin{corollary}
Let $\psi_F\in \CI(\overline{M})$ be a smooth function supported near $\pa_F\overline{M}$ such that $\left. \psi_F\right|_{\pa_F\overline{M}}$ is locally constant and consider the new $\Fcb$-metrics
$\widetilde{g}=e^{\psi_F}g$, $\widetilde{h}=e^{\psi_F}h$.  Then we have that
$$
        \det(\Delta_{\widetilde{g}},\Delta_{\widetilde{h}})= \det(\Delta_g,\Delta_h).
$$
\label{csp.4}\end{corollary}
\begin{proof}
We could use the Polyakov formula of \cite{AAR}, but this would require some extra decay behavior in $\psi_F$.  Instead, we simply apply the previous theorem twice by undoing and doing again a conformal surgery at each boundary component of $\pa_F\overline{M}$ to go from $g$ to $\widetilde{g}$ and from $h$ to $\widetilde{h}$.
\end{proof}

\section{Compactness of families of relatively isospectral surfaces} \label{sec:CsRelIso}

As mentioned in the introduction, Borthwick and Perry \cite{Borthwick-Perry} prove a compactness theorem for isoresonant metrics that coincide cocompactly and whose ends are hyperbolic funnels. Their proof of compactness, like the proof of compactness of Osgood-Phillips-Sarnak, uses the spectral assumption only through the equality of the relative heat invariants \eqref{RelHeatInv} and relative determinants \eqref{RelDet}. We restate their theorem with these assumptions.

\begin{theorem}[Borthwick-Perry \cite{Borthwick-Perry}] \label{ExtendedBP}
Let $(M_i,g_i)$ be a family of Riemannian surfaces that coincide cocompactly,
whose ends are hyperbolic funnels,
and assume that the relative heat invariants  and relative determinants  satisfy
\begin{equation*}
    a_k(g_i, g_j) = 0, \quad \det(\Delta_{g_i}, \Delta_{g_j}) = 1, \quad \Mforall i,j,k.
\end{equation*}
Then there is a Riemannian surface $(M, g_{\infty}),$ a subsequence $(M_{i_k},g_{i_k}),$ and a sequence of diffeomorphisms
\begin{equation*}
    \phi_k: M \lra M_{i_k},
    \Mwith \phi_{\ell}\circ \phi_{\ell'}^{-1}\rest{\cU_{\infty}} = \Id \text{ for any } \ell, \ell'
\end{equation*}
such that the metrics $\phi_k^*g_{i_k}$ converge to $g_{\infty}$ in $\CI.$
\end{theorem}

We can now finally give a proof of our main result.

\begin{proof}[Proof of Theorem~\ref{thm:RelIso}]
We need first to check that the various surfaces have the same topology.  This can be deduced from the relative heat invariants.  Namely, given two of the relatively isospectral surfaces $(M_i,g_i)$ and $(M_k,g_k)$, we can deform conformally $g_i$ and $g_k$ in $\cU_\infty$ in the same way to obtain metrics $\hat{g}_i$ and $\hat{g}_k$ also having vanishing relative heat invariants and such that,
\begin{itemize}
\item[(i)] the metrics $\hat{g}_i$ and $\hat{g}_k$ have no cusp, the cusps being removed by a conformal surgery at each marked point;
\item[(ii)] The metrics $\hat{g}_i$, $\hat{g}_k$ define incomplete metrics on $\overline{M}_i$ and $\overline{M}_k$ with boundary having no geodesic curvature.
\end{itemize}
Since these metrics $\hat{g}_i$, $\hat{g}_k$ have the same heat invariants, we know from \cite{McKean-Singer1967} that
$\overline{M}_i$ and $\overline{M}_k$ have the same Euler characteristic, so the same topology.  Thus, the isometry $M_i\setminus K_i\to M_k\setminus K_k$ can be extended to a diffeomorphism $M_i\to M_k$.  This means we can assume all the relatively isospectral metrics are defined on the same surface $M_\infty$ and agree on $\cU_\infty = M\setminus K$ where $K\subset M\setminus \pa M$ is a compact set.

Going back to the initial metrics $g_i$, we can, by doing a conformal surgery on $\cU_{\infty}$, remove all the cusps and transform each boundary into a funnel end hyperbolic near infinity.   We can also modify conformally the metrics in the remaining funnels to make them hyperbolic near infinity.  We thus get a new sequence of metrics $\widetilde{g}_i= e^{\psi} g_i$ on $\overline{M}$, where $\psi$ is a function supported on $\cU_{\infty}$ and $\overline{M}$ is obtained from $M$ by filling each cusp end with a point.      In doing so, the relative local heat invariants remains zero, and the relative determinant remains $1$ by Theorem~\ref{csp.3}.  We can thus apply Theorem~\ref{ExtendedBP} to find a Riemannian surface $(\overline{M}, \widetilde{g}_{\infty})$ and to extract a subsquence $\widetilde{g}_{i_k}$ and a sequence of diffeomorphisms $\phi_{i_k}: \overline{M} \to M_{i_k}$ such that $\phi_{i_k}^*g_{i_k}$ converges to $\widetilde{g}_{\infty}$ in $\CI(\overline{M}, \widetilde{g}_{\infty})$.  As explained in \cite{Borthwick-Perry}, we can choose the diffeomorphisms such that $\phi_{i_k}\circ\phi_{i_{k'}}^{-1}=\Id$ near each funnels of $(\overline{M},\widetilde{g}_{\infty})$.  The same argument can be carried out near each (filled) cusp, so that the diffeomorphisms can be chosen so that $\phi_{i_k}\circ\phi_{i_{k'}}^{-1}=\Id$ on $\cU_{\infty}$.  Undoing the conformal transformation on the metric $\widetilde{g}_{\infty}$ to obtain the metic $g_{\infty}= e^{-\psi}\widetilde{g}_{\infty}$ on $M$, we obtain the desired result with the subsequence $g_{i_k}$ and the diffeomorphisms $\phi_{i_k}$.
\end{proof}

\bibliography{Conformal_surgery}

\def\cprime{$'$} \def\cprime{$'$} \def\cdprime{$''$} \def\cprime{$'$}
  \def\cprime{$'$} \def\cprime{$'$} \def\cprime{$'$} \def\bud{$''$}
  \def\cprime{$'$} \def\cprime{$'$} \def\cprime{$'$} \def\cprime{$'$}
  \def\cprime{$'$} \def\cprime{$'$} \def\cprime{$'$} \def\cprime{$'$}
  \def\cprime{$'$} \def\cprime{$'$}
  \def\polhk#1{\setbox0=\hbox{#1}{\ooalign{\hidewidth
  \lower1.5ex\hbox{`}\hidewidth\crcr\unhbox0}}} \def\cprime{$'$}
  \def\cprime{$'$} \def\cprime{$'$} \def\cprime{$'$} \def\cprime{$'$}
  \def\cprime{$'$} \def\cprime{$'$} \def\cprime{$'$} \def\cprime{$'$}
  \def\cprime{$'$} \def\cprime{$'$} \def\cprime{$'$} \def\cprime{$'$}
  \def\cprime{$'$} \def\cprime{$'$} \def\cprime{$'$} \def\cprime{$'$}
  \def\cprime{$'$} \def\cprime{$'$} \def\cprime{$'$} \def\cprime{$'$}
  \def\cprime{$'$} \def\cprime{$'$} \def\cprime{$'$} \def\cprime{$'$}
  \def\cprime{$'$} \def\cprime{$'$}
\begin{thebibliography}{OPS88b}

\bibitem[AAR]{AAR}
Pierre Albin, Clara Aldana, and Fr\'ed\'eric Rochon.
\newblock Ricci flow and the determinant of the {L}aplacian on non-compact
  surfaces.
\newblock to appear in Comm. PDE.

\bibitem[Ald10]{Aldana:Isoresonant}
Clara Aldana.
\newblock Isoresonant conformal surfaces with cusps and boundedness of the
  relative determinant.
\newblock {\em Comm. Anal. Geom.}, 18(5):1009--1048, 2010.

\bibitem[And91]{Anderson}
Michael~T. Anderson.
\newblock Remarks on the compactness of isospectral sets in low dimensions.
\newblock {\em Duke Math. J.}, 63(3):699--711, 1991.

\bibitem[BG94]{Brooks-Glezen}
Robert Brooks and Paul Glezen.
\newblock An {$L^p$} spectral bootstrap theorem.
\newblock In {\em Geometry of the spectrum ({S}eattle, {WA}, 1993)}, volume 173
  of {\em Contemp. Math.}, pages 89--97. Amer. Math. Soc., Providence, RI,
  1994.

\bibitem[BJP03]{Borthwick-Judge-Perry:DetLap}
David Borthwick, Chris Judge, and Peter~A. Perry.
\newblock Determinants of {L}aplacians and isopolar metrics on surfaces of
  infinite area.
\newblock {\em Duke Math. J.}, 118(1):61--102, 2003.

\bibitem[BP11]{Borthwick-Perry}
David Borthwick and Peter~A. Perry.
\newblock Inverse scattering results for metrics hyperbolic near infinity.
\newblock 2011.

\bibitem[BPP92]{Brooks-Perry-Petersen}
Robert Brooks, Peter Perry, and Peter Petersen, V.
\newblock Compactness and finiteness theorems for isospectral manifolds.
\newblock {\em J. Reine Angew. Math.}, 426:67--89, 1992.

\bibitem[BPY89]{Brooks-Perry-Yang}
Robert Brooks, Peter Perry, and Paul Yang.
\newblock Isospectral sets of conformally equivalent metrics.
\newblock {\em Duke Math. J.}, 58(1):131--150, 1989.

\bibitem[Bun92]{Bunke1992}
U.~Bunke.
\newblock Relative {I}ndex {T}heory.
\newblock {\em J. Funct. Anal.}, 105:63--76, 1992.

\bibitem[Car02]{Carron:DetRel}
Gilles Carron.
\newblock D\'eterminant relatif et la fonction {X}i.
\newblock {\em Amer. J. Math.}, 124(2):307--352, 2002.

\bibitem[CGT82]{CGT1982}
J.~Cheeger, M.~Gromov, and M.~Taylor.
\newblock Finite propagation speed, kernel estimates for functions of the
  laplace operator, and the geometry of complete {R}iemannian manifolds.
\newblock {\em J. Differential Geom.}, 17:15--53, 1982.

\bibitem[CQ97]{Chang-Qing}
Sun-Yung~A. Chang and Jie Qing.
\newblock The zeta functional determinants on manifolds with boundary. {II}.
  {E}xtremal metrics and compactness of isospectral set.
\newblock {\em J. Funct. Anal.}, 147(2):363--399, 1997.

\bibitem[CX96]{Chen-Xu}
Roger Chen and Xingwang Xu.
\newblock Compactness of isospectral conformal metrics and isospectral
  potentials on a {$4$}-manifold.
\newblock {\em Duke Math. J.}, 84(1):131--154, 1996.

\bibitem[CY90]{Chang-Yang}
Sun-Yung~A. Chang and Paul C.-P. Yang.
\newblock Isospectral conformal metrics on {$3$}-manifolds.
\newblock {\em J. Amer. Math. Soc.}, 3(1):117--145, 1990.

\bibitem[Don87]{Donnelly1987}
Harold Donnelly.
\newblock Essential spectrum and heat kernel.
\newblock {\em J. Funct. Anal.}, 75:326--381, 1987.

\bibitem[Gil88]{Gilkey:Leading}
Peter~B. Gilkey.
\newblock Leading terms in the asymptotics of the heat equation.
\newblock In {\em Geometry of random motion ({I}thaca, {N}.{Y}., 1987)},
  volume~73 of {\em Contemp. Math.}, pages 79--85. Amer. Math. Soc.,
  Providence, RI, 1988.

\bibitem[GPS]{Gordon-Perry-Schueth}
Carolyn Gordon, Peter Perry, and Dorothee Schueth.
\newblock Isospectral and isoscattering manifolds: a survey of techniques and
  examples.
\newblock In {\em Geometry, spectral theory, groups, and dynamics}, volume 387
  of {\em Contemp. Math.}, pages 157--179.

\bibitem[Gui05]{Guillarmou:Mero}
Colin Guillarmou.
\newblock Meromorphic properties of the resolvent on asymptotically hyperbolic
  manifolds.
\newblock {\em Duke Math. J.}, 129(1):1--37, 2005.

\bibitem[GZ95]{Guillope-Zworski}
Laurent Guillop{\'e} and Maciej Zworski.
\newblock Upper bounds on the number of resonances for non-compact {R}iemann
  surfaces.
\newblock {\em J. Funct. Anal.}, 129(2):364--389, 1995.

\bibitem[HZ99]{Hassell-Zelditch}
Andrew Hassell and Steve Zelditch.
\newblock Determinants of {L}aplacians in exterior domains.
\newblock {\em Internat. Math. Res. Notices}, (18):971--1004, 1999.

\bibitem[Kim08]{Kim2008}
Young-Heon Kim.
\newblock Surfaces with boundary: their uniformizations, determinants of
  {L}aplacians, and isospectrality.
\newblock {\em Duke Math. J.}, 144(1):73--107, 2008.

\bibitem[Mel83]{Melrose:Drumheads}
Richard~B. Melrose.
\newblock Isospectral sets of drumheads are compact in
  {$\mathcal{C}^{\infty}$}.
\newblock Unpublished preprint from MSRI, 048-83 Available online at
  {http://math.mit.edu/$\sim$rbm/paper.html}, 1983.

\bibitem[Mel94]{Melrose:Scattering}
Richard~B. Melrose.
\newblock Spectral and scattering theory for the {L}aplacian on asymptotically
  {E}uclidian spaces.
\newblock In {\em Spectral and scattering theory ({S}anda, 1992)}, volume 161
  of {\em Lecture Notes in Pure and Appl. Math.}, pages 85--130. Dekker, New
  York, 1994.

\bibitem[MM87]{Mazzeo-Melrose:Zero}
Rafe Mazzeo and Richard~B. Melrose.
\newblock Meromorphic extension of the resolvent on complete spaces with with
  asymptotically negative curvature.
\newblock {\em J. Funct. Anal.}, pages 260--310, 1987.

\bibitem[MS67]{McKean-Singer1967}
H.P. McKean and I.M. Singer.
\newblock Curvature and the eigenvalues of the laplacian.
\newblock {\em J. Differential Geometry}, 1:43--69, 1967.

\bibitem[OPS88a]{OPS4}
Brian Osgood, Ralph Phillips, and Peter Sarnak.
\newblock Compact isospectral sets of plane domains.
\newblock {\em Proc. Nat. Acad. Sci. U.S.A.}, 85(15):5359--5361, 1988.

\bibitem[OPS88b]{OPS2}
Brian Osgood, Ralph Phillips, and Peter Sarnak.
\newblock Compact isospectral sets of surfaces.
\newblock {\em J. Funct. Anal.}, 80(1):212--234, 1988.

\bibitem[OPS88c]{OPS1}
Brian Osgood, Ralph Phillips, and Peter Sarnak.
\newblock Extremals of determinants of {L}aplacians.
\newblock {\em J. Funct. Anal.}, 80(1):148--211, 1988.

\bibitem[OPS89]{OPS3}
Brian Osgood, Ralph Phillips, and Peter Sarnak.
\newblock Moduli space, heights and isospectral sets of plane domains.
\newblock {\em Ann. of Math.}, 129:293--362, 1989.

\bibitem[Zho97]{Zhou}
Gengqiang Zhou.
\newblock Compactness of isospectral compact manifolds with bounded curvatures.
\newblock {\em Pacific J. Math.}, 181(1):187--200, 1997.

\end{thebibliography}
\bibliographystyle{alpha}

\end{document}